\documentclass{siamltex}

\usepackage{fullpage}
\usepackage{comment}
\usepackage{upgreek}
\usepackage{cite}
\usepackage[mathscr]{eucal}
\usepackage{color}
\usepackage[usenames,dvipsnames,svgnames,table]{xcolor}
\usepackage{amsmath,amssymb,amsopn,mathtools}
\usepackage{graphicx}
\usepackage{hyperref}       
\hypersetup{
	colorlinks=true,   
	citecolor=blue,     
	filecolor=blue,     
	linkcolor=red,    
	urlcolor=blue}      

\usepackage{relsize}
\numberwithin{equation}{section}
\usepackage{enumitem}

\usepackage{latexsym,subfigure,color,IEEEtrantools}     
\usepackage{stmaryrd}
\usepackage{multirow}
\usepackage{algorithm,algorithmic}

\newcommand\norm[1]{\left\|#1\right\|}
\newcommand\abs[1]{\lvert#1\rvert}
\newcommand\Abs[1]{\left\lvert#1\right\rvert}
\newcommand{\tforall}{\text{ for all }}

\newcommand{\ms}{\text{ms}}
\newcommand{\aux}{\text{aux}}
\newcommand{\on}{\text{on}}
\newcommand{\glo}{\text{glo}}
\newcommand{\dg}{\text{DG}}
\newtheorem{example}[theorem]{Example}

\DeclareMathOperator{\spa}{span}

\begin{document}

\title{Online adaptive algorithm for Constraint Energy Minimizing Generalized Multiscale Discontinuous Galerkin Method}
\author{
Sai-Mang Pun\thanks{Department of Mathematics, Texas A\&M University, College Station, TX 77843, USA (\texttt{smpun@math.tamu.edu})}
\and Siu Wun Cheung\thanks{Center for Applied Scientific Computing, Lawrence
  Livermore National Laboratory, Livermore, CA 94550, USA (\texttt{cheung26@llnl.gov})}
}
\maketitle

\begin{abstract}
In this research, we propose an online basis enrichment strategy within the framework of a recently developed constraint energy minimizing generalized multiscale discontinuous Galerkin method (CEM-GMsDGM). 
Combining the technique of oversampling, 
one makes use of the information of the current residuals to adaptively construct basis functions in the online stage to reduce the error of multiscale approximation. 
A complete analysis of the method is presented, which shows the proposed online enrichment leads to a fast convergence from multiscale approximation to the fine-scale solution. 
The error reduction can be made sufficiently large by suitably selecting oversampling regions and the number of oversampling layers. 
Further, the convergence rate of the enrichment algorithm depends on a factor of exponential decay regarding to the number of oversampling layers and a user-defined parameter. 
Numerical results are provided to demonstrate the effectiveness and efficiency of the proposed online adaptive algorithm. 
\end{abstract}

\section{Introduction}
\label{sec:intro}
Physical modeling in heterogeneous media with multiple scales and high contrast 
is the heart of many scientific and engineering applications. 
In many cases, the underlying mathematical model does not possess closed-form analytic solutions. 
Extensive research effort has been devoted to develop computational methods for 
obtaining numerical solutions from simulation. 
Mesh-based methods, such as finite difference, finite volume and finite element methods, 
have been well-studied and widely used for numerical modeling in various engineering applications. 
In recent years, the development of discontinuous Galerkin (DG) method has been very active
in fluid dynamics \cite{conv-diff,cockburn05,riviere2008discontinuous,sdg-ns1} 
and wave propagations \cite{ipdg-wave,newdg,newdg1,meta}.
In contrary to conforming Galerkin (CG) finite element methods, 
DG methods make use of piecewise basis functions for achieving 
better conservation properties in convection-dominated problems and wave propagations.
While the development of these numerical schemes has become very mature 
and rigorous mathematical theory has been built for justification of these methods, 
straightforward application of these numerical solvers are not efficient 
for highly heterogeneous problems, since a very fine grid is needed to capture all the 
heterogeneities in the physical properties and essential for obtaining accurate numerical solutions. 
Traditional numerical methods may then become prohibitively expensive and even unfeasible. 

To remedy this situation, the development of efficient computational multiscale methods 
for solving multiscale problems at reduced computational expense has been of great interest to various 
scientific and engineering disciplines. 
Existing approaches include numerical homogenization approaches \cite{weh02}, 
multiscale finite element methods (MsFEM) \cite{ch03,cgh09,eh09,ehw99,hw97}, 
variational multiscale methods (VMS) \cite{calo2011note,hfmq98,hughes2007variational,Iliev_MMS_11}, 
heterogeneous multiscale methods (HMM) \cite{abdulle05,ee03,emz05}, 
and generalized multiscale finite element methods (GMsFEM) 
\cite{chung2015generalizedwave,chung2016adaptive,chung2014adaptive,egh12}. 
The central idea of these multiscale methods is to construct coarse-scale numerical solvers 
which typically seek for solution on a coarse grid with much fewer degrees of freedom 
than the fine grid that is used to capture all the heterogeneities in the medium properties. 
In numerical homogenization approaches, effective properties are computed on the coarse grid 
and used to formulate the model problem and therefore the numerical solver. 
While these approaches are simple, they are limited to the assumption that the multiple scales in
the medium properties can be separated. 

Meanwhile, the goal of multiscale methods is to incorporate the fine-scale effects in 
the degrees of freedom used to formulate the global problem. 
It is therefore important to make sure the degrees of freedom are adequate for 
representability of solution in the multiscale media. 
Many multiscale methods in the literature, including MsFEM, VMS, and HMM, 
construct one degree of freedom for each coarse region to handle the effects of local heterogeneities. 
For numerical modeling of convection-dominated problems and wave propagations in heterogeneous media, 
multiscale methods in the DG framework have been investigated 
\cite{ehw99,buffa2006analysis,eglmsMSDG,AADA,
elfverson2013dg,efendiev2015spectral,chung2017dg,chung2018dg}. 
In these approaches, multiscale basis functions are in general discontinuous on the coarse grid, 
and stabilization or penalty terms are added to ensure well-posedness of the global problem. 

While these methods had drawn lots of attention and been successfully applied in various multiscale problems, 
multiple multiscale basis functions are necessary in order to accurately represent the local features of the solution 
for more complex multiscale problems in which each local coarse region contains several high-conductivity regions.
GMsFEM employs the idea of model reduction to extract local dominant modes and 
identify the underlying low-dimensional local structures for solution representation in each coarse region. 
This allows systematic enrichment of the coarse-scale space with fine-scale information. 
By including multiple degrees of freedom in each coarse region, 
the error of GMsFEM is related to the smallest eigenvalues which are excluded in the local spectral problems. 
For a more detailed discussion on GMsFEM, we refer the readers to 
\cite{chung2016adaptive,chung2015residual,chung2014adaptive,
egh12,eglp13oversampling,egw10} 
and the references therein. 

For classical numerical schemes, such as the finite element and finite difference methods,
the solution accuracy is subject to the convergence of the mesh size. 
Moreover, the convergence should be independent of these physical parameters. 
However, for multiscale problems, it is difficult to adjust coarse-grid mesh size based on scales and contrast, 
making deriving multiscale methods with convergence on coarse mesh size and independent of scales and contrast
a non-trivial task. Very recently, several multiscale methods with mesh convergence  
had been developed using localization techniques 
\cite{maalqvist2014localization, owhadi2017multigrid,owhadi2014polyharmonic}. 
This idea had been studied and extended to multiple degrees of freedom per coarse region for 
\cite{hou2017sparse,chung2018constraint,chung2018mixed,cheung2018mc}. 

Our work is built within the framework of a class of recently developed mutliscale methods, 
namely the constraint energy minimizing generalized multiscale finite element method (CEM-GMsFEM), 
which exhibits both coarse mesh convergence and spectral convergence. 
More precisely, CEM-GMsFEM is considered within a DG discretization setting \cite{cheung2020constraint} 
and extended to wave propagation in heterogeneous and high contrast media \cite{cheung2020wave}. 
Local spectral problems and constraint energy minimization problems are used to construct multiple 
multiscale DG basis functions per coarse region, which are then coupled to formulate a global coarse-scale system 
of equations using the interior penalty discontinuous Galerkin (IPDG) formulation. 
This solution scheme is referred to the offline multiscale method in this paper and 
is used to initialize an adaptive solution procedure in an online stage. 
Iteratively, the multiscale solution and the data are used to compute the residual information
which suggests online basis functions to be included to enrich the multiscale space and improve the accuracy of the solution. 
We remark that the solution accuracy of CEM-GMsDGM is important as it is applied to 
multiscale convection-dominated problems and wave propagation in heterogeneous and high-contrast media. 

In this research, we develop and analyze an online enrichment strategy for CEM-GMsFEM within the DG setting. The strategy is based on the information of local residuals and the technique of oversampling, adopting the ideas presented in \cite{chung2018fast} with CG setting and in \cite{chung2019online} for mixed formulation. As a result, the corresponding online basis functions are supported in some oversampled regions. This construction differs from the previous online approach in \cite{chung2015residual} since CEM-GMsDGM makes use of the technique of oversampling. In particular, the online basis functions are formulated in the oversampled regions. 
We show that the convergence rate depends on the factor of exponential decay and a user-defined parameter of the online adaptive enrichment. 
One obtains accurate approximation in a few online iterations by choosing appropriate number of oversampling layers. 

The paper is organized as follows. In Section~\ref{sec:prelim}, we will introduce the notions of grids, 
and essential discretization details such as DG finite element spaces and IPDG formulation on the coarse grid. 
We will then briefly review the construction of offline multiscale space in Section~\ref{sec:offline_method}. 
The online adaptive method will be presented in Section~\ref{sec:online_method} 
and analyzed in Section~\ref{sec:analysis}.
Numerical results will be provided in Section~\ref{sec:numerical} to demonstrate the effectiveness and efficiency of the proposed online adaptive algorithm.
Concluding remarks will be given in Section~\ref{sec:conclusion}.

\section{Preliminaries}
\label{sec:prelim}
We consider the following high-contrast flow problem: 
\begin{equation}
\begin{split}
-\nabla \cdot \left(\kappa \nabla u\right) &= f \quad \text{in} ~ \Omega,\\
u & = 0 \quad \text{on} ~ \partial \Omega.
\end{split}
\label{eq:elliptic}
\end{equation}
Here, the set $\Omega \subset \mathbb{R}^d$ ($ d\in \{ 2, 3 \}$) is a computational domain and 
$f \in L^2(\Omega)$ is a given source term. 
We assume that the permeability field $\kappa \in L^\infty (\Omega) $ is highly heterogeneous 
such that there exist two constants $0 <  \kappa_0 \ll \kappa_1 < + \infty$ such that 
$\kappa_0 \leq \kappa(x) \leq \kappa_1$ for almost every $x \in \Omega$. 

Next, we introduce the notions of coarse and fine meshes. 
We start with a usual partition $\mathcal{T}^H$ of $\Omega$ into finite elements, 
which does not necessarily resolve any multiscale features. 
The partition $\mathcal{T}^H$ is called a coarse grid and 
a generic element $K$ in the partition $\mathcal{T}^H$ is called a coarse element. 
Moreover, $\displaystyle{H := \max_{K \in \mathcal{T}^H } \left ( \max_{x, y \in K} \abs{x-y} \right ) > 0}$ is called the coarse mesh size.
We let $N_c$ be the number of coarse grid nodes and 
$N$ be the number of coarse elements. 
We also denote $\mathcal{E}^H$ the collection of all coarse grid edges. 
We perform a refinement of $\mathcal{T}^H$ to obtain a fine grid $\mathcal{T}^h$, 
where $\displaystyle{h := \max_{\tau \in \mathcal{T}^h } \left ( \max_{x, y \in \tau} \abs{x-y} \right ) > 0}$ is the mesh size of the fine grid. 
It is assumed that the fine grid is sufficiently small to resolve the heterogeneities. 
An illustration of the fine and coarse grids and a coarse element is shown in Figure~\ref{fig:mesh}. 

\begin{figure}[ht!]
\centering
\includegraphics[width=0.45\linewidth]{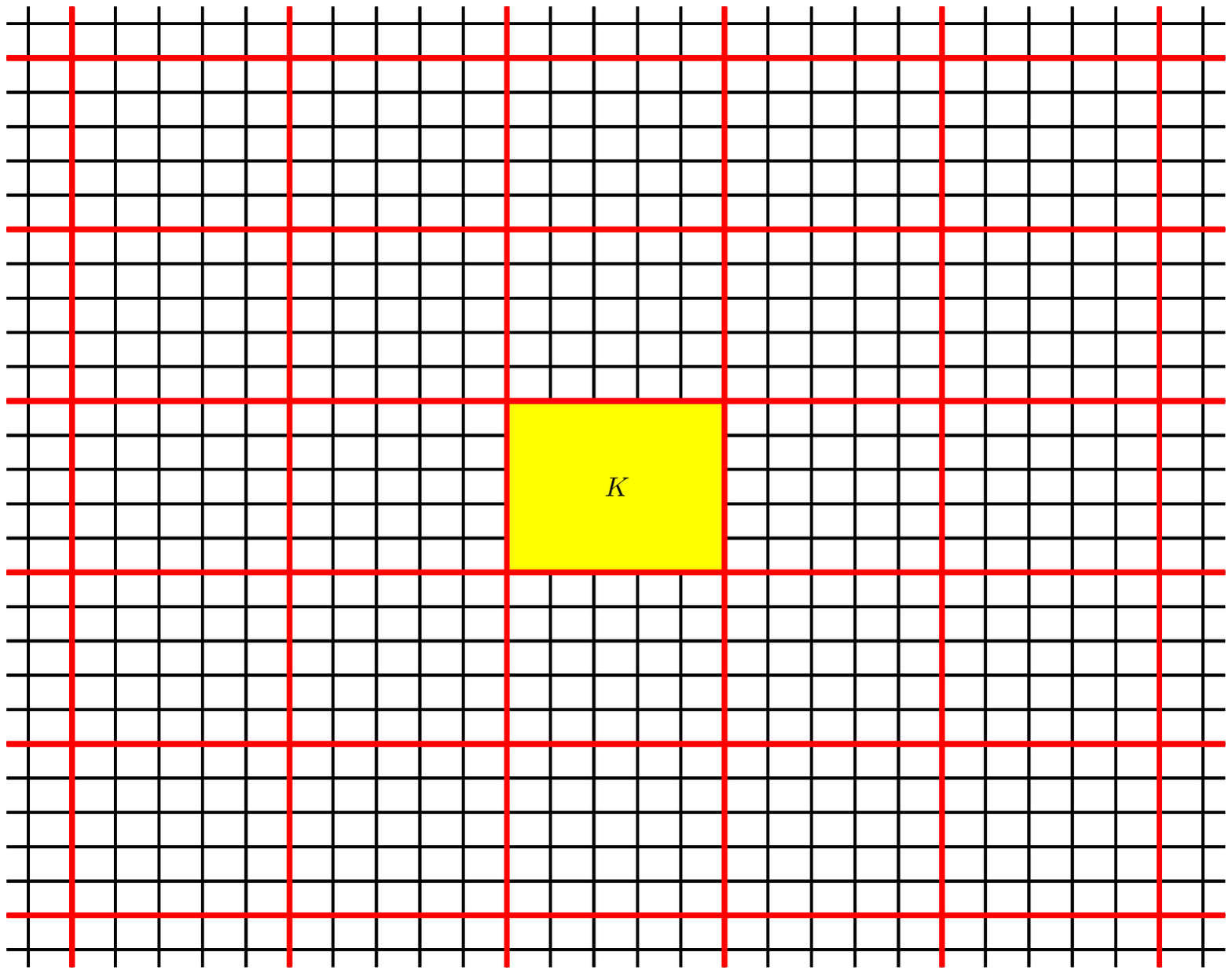}
\caption{An illustration of partition of computational domain.}
\label{fig:mesh}
\end{figure}

In this work, we consider the discontinuous Galerkin discretization 
and the interior penalty discontinuous Galerkin global formulation. 
For the $i$-th coarse block $K_i \in \mathcal{T}^H$, we denote $V(K_i)$ the restriction of the 
Sobolev space $V := H_0^1(\Omega)$ on $K_i$. Let $V_h(K_i)$ be 
the conforming bilinear elements defined on the fine grid $\mathcal{T}^h$ in $K_i$, i.e.
\begin{equation}
V_h(K_i) := \left\{ v \in V(K_i): v \vert_\tau \in \mathbb{Q}^1(\tau) \text{ for all } 
\tau \in \mathcal{T}^h \text{ and } \tau \subset K_i\right\}, 
\end{equation}
where $\mathbb{Q}^1(\tau)$ stands for the bilinear element on the fine grid block $\tau$. 
The DG approximation space is then given by the space of 
coarse-scale locally conforming piecewise bilinear 
fine-grid basis functions, namely 
\begin{equation}
V_h := \bigoplus_{i=1}^N V_h(K_i). 
\end{equation}
We remark that functions in $V_h$ are continuous within coarse blocks, 
but discontinuous across the coarse grid edges in general. 
Given a subdomain $\Omega' \subseteq \Omega$ formed by a union of coarse blocks $K \in \mathcal{T}^H$, 
we also define the local DG approximate space by 
$$ V_h(\Omega') := \bigoplus_{K_i \subset \Omega'} V_h(K_i).$$

The global formulation of IPDG method then reads:
find $u_h \in V_h$ such that
\begin{equation}
a_{\dg}\left(u_h ,w\right) = \int_{\Omega} fw \, dx \quad \text{ for all } w \in V_h, 
\label{eq:sol_dg}
\end{equation}
where the bilinear form $a_{\dg}$ is defined by:
\begin{equation}
\begin{split}
a_{\dg}\left(v,w\right) 
& := \sum_{K \in \mathcal{T}^H} \int_K \kappa \nabla v \cdot \nabla w \, dx 
- \sum_{E \in \mathcal{E}^H} \int_E \{ \kappa \nabla v \cdot n_E \} \llbracket w \rrbracket \, d\sigma \\
& \quad - \sum_{E \in \mathcal{E}^H} \int_E \{ \kappa \nabla w \cdot n_E \} \llbracket v \rrbracket \, d\sigma 
+ \dfrac{\gamma}{h} \sum_{E \in \mathcal{E}^H} \int_E \overline{\kappa} \llbracket v \rrbracket \llbracket w \rrbracket \, d\sigma. 
\end{split}
\label{eq:dg_bilinear}
\end{equation}
The scalar $\gamma > 0$ is a penalty parameter and 
$n_E$ is a fixed unit normal vector defined on the coarse edge $E \in \mathcal{E}^H$. 
Note that the average and the jump operators in \eqref{eq:dg_bilinear}
are defined in the classical way. 
Specifically, consider an interior coarse edge $E \in \mathcal{E}^H$ and 
let $K^+$ and $K^-$ be the two coarse grid blocks sharing the edge $E$, 
where the unit normal vector $n_E$ is pointing from $K^+$ to $K^-$. 
For a piecewise smooth function $G$ with respect to the coarse grid $\mathcal{T}^H$, we define
\begin{equation}
\{ G \}  := \dfrac{1}{2}\left(G^+ + G^-\right) \quad \text{and} \quad \llbracket G \rrbracket  := G^+ - G^-,
\end{equation}
where $G^+ := G\vert_{K^+}$ and $G^- := G\vert_{K^-}$. 
Moreover, we define 
$$\overline{\kappa} := \frac{1}{2} \left(\kappa_{K^+} + \kappa_{K^-} \right)$$
on the edge $E$,  
where $\kappa_{K^\pm}$ is the maximum value of $\kappa$ over $K^\pm$, respectively. 
For a coarse edge $E$ lying on the boundary $\partial \Omega$, we define
$\{G\}:=G$, $\llbracket  G \rrbracket :=G$, and $\overline{\kappa}:=\kappa_K$ on $E$, 
where we always assume that $n_E$ is pointing outside of the domain $\Omega$. 
We define the energy norm on the space of coarse-grid piecewise smooth functions by
\begin{equation}
\norm{w}_a := \sqrt{a_{\dg}(w,w)}.
\end{equation}
We also define the DG-norm by 
\begin{equation}
\norm{w}_{\dg} :=\left ( \sum_{K \in \mathcal{T}^H} \int_K \kappa \vert \nabla w \vert^2 \, dx 
+ \dfrac{\gamma}{h} \sum_{E \in \mathcal{E}^H} \int_E \overline{\kappa} \llbracket w \rrbracket^2 \, d\sigma \right  )^{1/2}.
\end{equation}
The two norms are equivalent on the subspace of coarse-grid piecewise bi-cubic polynomials, that is, 
there exists $C_0 \geq 1$ such that 
\begin{equation}
C_0^{-1} \| w \|_{a} \leq \| w \|_{\dg} \leq C_0 \| w \|_a.
\end{equation}
The continuity and coercivity results of the bilinear form $a_{\dg}(\cdot,\cdot)$ 
with respect to the DG-norm is ensured by a sufficiently large penalty parameter $\gamma$.
While the method works well for general highly heterogeneous field $\kappa$, 
we assume $\kappa$ is piecewise constant on the fine grid $\mathcal{T}^h$ 
for the sake of simplicity in our analysis. 

\section{Offline multiscale method}
\label{sec:offline_method}
In this section, we briefly present the construction of the multiscale basis functions. 
We use the concept of GMsFEM to construct our auxiliary multiscale basis functions 
on a generic coarse block $K$ in the coarse grid. 
We consider $V_h(K_i)$ as the snapshot space related to $K_i$ and 
we perform a dimension reduction through a spectral problem, 
which is to find a real number $\lambda_j^{\left(i\right)} \in \mathbb{R}$ 
and a function $\phi_j^{\left(i\right)} \in V_h(K_i)$ such that
\begin{equation}
a_i\left(\phi_j^{\left(i\right)}, w\right) = \lambda_j^{\left(i\right)} s_i\left(\phi_j^{\left(i\right)}, w\right) \quad \text{ for all } w \in V_h(K_i). 
\label{eq:spectral_prob}
\end{equation}
Here, $a_i(\cdot,\cdot)$ is a symmetric non-negative definite bilinear operator 
and $s_i(\cdot,\cdot)$ is a symmetric positive definite bilinear operator defined on $V_h(K_i) \times V_h(K_i)$. 
We remark that the above problem is solved on the fine mesh in actual computations. 
Based on our analysis, we can choose
\begin{equation}
\begin{split}
a_i\left(v,w\right) & = \int_{K_i} \kappa \nabla v \cdot \nabla w \, dx, \\
s_i\left(v,w\right) & = \int_{K_i} \tilde{\kappa} v w \, dx,
\end{split}
\label{eq:spectral_bilinear_form}
\end{equation}
where $\tilde{\kappa} := \sum_{j=1}^{N_c} \kappa \vert \nabla \chi_j \vert^2$ and 
$\{\chi_j \}_{j=1}^{N_c}$ is a set of partition of unity functions. 
Let $\{ \lambda_j^{\left(i\right)} \}$ be the set of eigenvalues of \eqref{eq:spectral_prob} 
arranged in ascending order in $j$. We use the first $L_i$ eigenfunctions, corresponding to the first smallest $L_i$ eigenvalues, 
to construct our local auxiliary multiscale space
$V_{\text{aux}}^{\left(i\right)} := \text{span} \{ \phi_j^{\left(i\right)}: 1 \leq j \leq L_i\}$. 
The global auxiliary multiscale space $V_{\text{aux}}$ is then defined as 
the direct sum of these local auxiliary multiscale spaces
\begin{equation}
V_{\text{aux}} = \bigoplus_{i=1}^N V_{\text{aux}}^{\left(i\right)}.
\end{equation}
The bilinear form $s_i(\cdot,\cdot)$ in \eqref{eq:spectral_bilinear_form} defines an inner product 
with the induced norm $\|v\|_{s(K_i)} := \sqrt{s_i \left(v,v\right)}$. 
These local inner products and norms provide natural definitions of inner
product and norm, 
which are defined by
\begin{equation}
\begin{split}
s\left(v,w\right) & := \sum_{i=1}^N s_i\left(v,w\right),\\ 
\norm{v}_s & := \sqrt{s\left(v,v\right)}. 
\end{split}
\end{equation}
Before we move on to discuss the construction of multiscale basis functions, 
we introduce some tools which will be used to describe our method and analyze the convergence.
We introduce a projection operator $\pi: V_h \to V_{\text{aux}}$ by 
$\pi := \sum_{i=1}^N \pi_i$, where
\begin{equation}
\pi_i(v) := \sum_{j=1}^{L_i} \dfrac{s_i\left(v,\phi_j^{\left(i\right)}\right)}{s_i\left(\phi_j^{\left(i\right)},\phi_j^{\left(i\right)}\right)} \phi_j^{\left(i\right)} \quad \text{ for all } v \in V_h, \text{ for all } i = 1,2,\cdots,N.
\end{equation}
We remark that due to the $s_i$-orthogonality of the auxiliary basis functions, the projection operator satisfies the following inequality
\begin{eqnarray}
\norm{(1 - \pi)v}_{s}^2 \leq \Lambda^{-1} \sum_{i=1}^N a_i(v,v)
\label{eqn:lambda-1}
\end{eqnarray}
for any $v \in V_h$, where $\Lambda := \min_{1 \leq i \leq N} \lambda_{L_i +1}^{(i)}$. 

Next, we present the construction of the (offline) multiscale basis functions in $V_h$. 
For each auxiliary function $\phi_j^{(i)}$, we construct a multiscale basis function $\psi_{j,\text{ms}}^{(i)}$ whose support is $K_i^+ = K_{i,m}$, where 
$$ K_{i,m} := \left \{ \begin{array}{ll}
K_i & \text{if} ~ m = 0, \\
\bigcup \{ K : K \cap K_{i,m-1} \neq \emptyset \} & \text{if} ~ m \geq 1,
\end{array} \right .$$
for any nonnegative integer $m \in \mathbb{N}$. 
The multiscale basis function $\psi_{j,\text{ms}}^{(i)} \in V_{h}(K_i^+)$ is defined to be the solution of the following system: 
\begin{equation}
a_{\dg}\left(\psi_{j,\text{ms}}^{\left(i\right)}, \psi\right) + s\left(\pi\left(\psi_{j,\text{ms}}^{\left(i\right)}\right), \pi\left(\psi\right)\right) = s\left(\phi_j^{\left(i\right)},\pi\left(\psi\right)\right) \quad \text{ for all } \psi \in V_h\left(K_{i}^+\right).
\label{eq:var2}
\end{equation}
We remark that the authors in \cite{cheung2020constraint} define the multiscale basis functions as 
minimizers of constraint energy minimization problem accompanying with the $s_i$-orthogonality of auxiliary functions. 
In the contrary, we use the modified definition \eqref{eq:var2} of multiscale basis functions with the 
so-called relaxed formulation (see \cite[Section 6]{chung2018constraint} for the result within continuous Galerkin setting). 
It is important to note that the multiscale basis functions are localized in the sense that they are supported 
in an oversampled region $K_{i,m}$ and approximate the corresponding 
global basis function which exhibits a property of exponential decay. 
As suggested by our analysis, the localization provides the same convergence rate with respect to the coarse mesh size 
and can be solved with reduced expense thanks to the localized support. 
With the multiscale basis functions constructed, we define the multiscale DG finite element space as 
\begin{equation}
V_{\text{off}} := \text{span}  \{\psi_{j,\text{ms}}^{\left(i\right)} : 1 \leq j \leq L_i, 1 \leq i \leq N \},
\label{eqn:local_ms_space}
\end{equation}
which is a subspace of $V_h$. 
After the multiscale DG finite element space is constructed, 
the offline multiscale solution $u_{\text{off}}$ is given by: find $u_{\text{off}} \in V_{\text{off}}$ such that
\begin{equation}
a_{\dg}\left(u_{\text{off}},w\right) = \int_{\Omega} fw \, dx \quad \text{ for all } w \in V_{\text{off}}. 
\label{eq:sol_ms}
\end{equation}

\section{Online adaptive algorithm}
\label{sec:online_method}
In this section, we develop an online adaptive algorithm for the CEM-GMsDGM. We first present the construction of the online basis function. Based on this construction, we propose an online adaptive algorithm. 

\subsection{Online basis functions}
We present the construction of the online basis function. First, we define the residual functional $r: V_h \to \mathbb{R}$. Let $u_{\text{app}} \in V_h$ be a numerical approximation. The residual functional is defined to be 
\begin{equation}
r(v) := a_{\dg}(u_{\text{app}}, v) - \int_\Omega fv \, dx \quad \text{ for all } v \in V_h.
\label{eqn:glo_residual}
\end{equation}
We also consider local residuals. For each coarse node $x_i$, we define a coarse neighborhood $\omega_i := \bigcup \{ K: x_i \in K, ~ K \in \mathcal{T}^H \}$. For each coarse neighborhood $\omega_i$, we define the local residual functional $r_i: V_h \to \mathbb{R}$ such that 
\begin{equation}
r_i(v) := a_{\dg}(u_{\text{app}}, \chi_i v) - \int_\Omega f \chi_i v \, dx \quad \text{ for all } v \in V_h.
\label{eqn:loc_residual}
\end{equation}
The set of functions $\{ \chi_i \}_{i=1}^{N_c}$ forms a partition of unity with respect to the coarse grid. We remark that one can take $\{ \chi_i \}_{i=1}^{N_c}$ to be the set of  standard multiscale basis functions or the standard piecewise linear functions. 

Next, we define the online basis function. The construction of the online basis function is related to the local residual. 
For any coarse neighborhood $\omega_i$ related to the coarse node $x_i$, we define $\omega_i^+ = \omega_{i,m}$ such that 
$$ \omega_{i,m} := \left \{ \begin{array}{ll}
\omega_i & m = 0, \\
\bigcup \{ K: \omega_{i,m-1} \cap K \neq \emptyset \} & m \geq 1,
\end{array} \right .$$
for any nonnegative integer $m \in \mathbb{N}$. We denote $m$ the number of oversampling layers. 
Using the local residual, one can define the online basis function $\beta_{\text{on}}^{(i)} \in V_h (\omega_i^+)$ whose support is the oversampled region $\omega_i^+$. 
More precisely, the online basis function $\beta_{\text{on}}^{(i)} \in V_h(\omega_i^+)$ is defined to be the solution of the following cell problem: 
\begin{equation}
a_{\dg} (\beta_{\text{on}}^{(i)}, v) + s(\pi (\beta_{\text{on}}^{(i)}), \pi(v)) = r_i(v) \quad \text{ for all } v \in V_h(\omega_i^+).
\label{eqn:loc_online_basis}
\end{equation}

\subsection{Online adaptive enrichment}
In this section, we present an online adaptive algorithm with enrichment of online basis functions defined in the previous section. Once the online basis functions are constructed, we include those newly constructed functions into the multiscale space. With this enriched space, we can compute a new numerical solution by solving the equation \eqref{eq:sol_ms}. One can repeat the process to enrich the multiscale space until the residual norm is smaller than a prescribed tolerance. 

First, we set the initial multiscale space to be $V_{\ms}^{(0)} := V_{\text{off}}$, where $V_{\text{off}}$ is defined in \eqref{eqn:local_ms_space}. 
Next, we choose a parameter $\theta \in [0, 1)$, where it determines the number of online basis functions that are included in the space during each iteration. 
The online adaptive algorithm sequentially defines residual functionals by taking $u_{\text{app}} = u_{\ms}^{(k)}$ in \eqref{eqn:glo_residual} and \eqref{eqn:loc_residual} with $u_{\ms}^{(k)}$ being the solution of \eqref{eq:sol_ms} over the multiscale space $V_{\ms}^{(k)}$. That is, we define $r^k$ the global residual operator at $k$-th level of enrichment such that 
$$ r^k(v) := a_{\dg} (u_{\ms}^{(k)}, v) - \int_\Omega fv ~ dx \quad \tforall v \in V_h.$$
This online adaptive method enriches the multiscale space 
$V_{\ms}^{(k)} \subset V_{\ms}^{(k+1)}$ by adding online basis functions in \eqref{eqn:loc_online_basis}, and 
generates an updated multiscale solutions $u_{\ms}^{(k+1)}$ in $V_{\ms}^{(k+1)}$ by Galerkin projection. 
The complete procedure of the online adaptive algorithm is listed in Algorithm \ref{algo:online}. 

\begin{algorithm}[ht]
	\caption{Online adaptive enrichment algorithm}
	\begin{algorithmic}[1]
	\STATE {\bf Input:} A given source function $f\in L^2(\Omega)$, a parameter $\theta \in [0,1)$, a set of numbers $\{ L_i \}_{i=1}^N$, a number of oversampling layers $m$, and a number of iterations $\texttt{NIter} >0$. 
	\STATE Construct the offline space $V_{\text{off}}$. 
	\STATE Set $k = 0$, $V_{\ms}^{(k)} = V_{\text{off}}$, and $u_{\ms}^{(k)} = u_{\text{off}}$ obtained in \eqref{eq:sol_ms}. 
	\FOR{$k =0$ to \texttt{NIter}-1}
		\STATE For each $i \in \mathcal\{ 1, \cdots, N_c \}$, compute $\delta_i^k := \norm{z_i^k}_{a^*}$, where 
		$$z_i^k(v) := a_{\dg}(u_{\ms}^{(k)}, v) - \int_\Omega fv \; dx \quad \text{ for all } v \in V_h(\omega_i), \quad \norm{z_i^k}_{a^*} := \sup_{v \in V_h(\omega_i)} \frac{\abs{r^k(v)}}{\norm{v}_a}.$$
		\STATE Enumerate the indices of $\omega_i$ such that $\delta_1^k \geq \delta_2^k \geq \cdots \geq \delta_{N_c}^k$. 
		\STATE Find the smallest integer $p = p(k)\in \mathbb{N}$ such that 
		\begin{equation}
		\sum_{i=p+1}^{N_c} (\delta_i^k)^2 < \theta \sum_{i=1}^{N_c} (\delta_i^k)^2.
		\label{eqn:theta_percent}
		\end{equation}
		\STATE For each $i \in \{1,\cdots, p\}$, compute the online basis function $\beta_{\text{on}}^{(i,k)}$ by solving 
		$$a_{\dg} (\beta_{\text{on}}^{(i,k)}, v) + s(\pi (\beta_{\text{on}}^{(i,k)}), \pi(v)) = r_i^k(v) \quad \text{ for all } v \in V_h(\omega_i^+)$$
		with $r_i^k(v) := a_{\dg}(u_{\ms}^{(k)}, \chi_i v) - \int_\Omega f \chi_i v \; dx$ for all $v \in V_h(\omega_i)$. 
		\STATE Set $V_{\ms}^{(k+1)} = V_{\ms}^{(k)} \bigoplus \spa\left \{ \beta_{\text{on}}^{(i,k)} \right \}_{i=1}^p$. 
		\STATE Solve $u_{\ms}^{(k+1)} \in V_{\ms}^{(k+1)}$ such that 
		\begin{equation}
		a_{\dg}\left(u_{\text{ms}}^{(k+1)},w\right) = \int_{\Omega} fw \, dx \quad \text{ for all } w \in V_{\text{ms}}^{(k+1)}. 
		\end{equation}
	\ENDFOR
	\STATE {\bf Output:} A sequence of multiscale solutions $\left \{ u_{\ms}^{(k)} \right \}_{k=0}^{\texttt{NIter}}$. 
	\end{algorithmic}
	\label{algo:online}
\end{algorithm}

\section{Error analysis}\label{sec:analysis}
In this section, we analyze the convergence rate of the proposed online adaptive algorithm. First, we need to introduce some notations. 
Given a subdomain $\Omega' \subseteq \Omega$ formed by a union of coarse blocks $K \in \mathcal{T}^H$, 
we define the local $s$-norm by 
$$\norm{w}_{s\left(\Omega'\right)} 
:= \left ( \sum_{K \subseteq \Omega'} \int_K \tilde{\kappa} \vert w \vert^2 \, dx \right )^{1/2}.$$ 

Next, we recall some theoretical results in \cite{cheung2020constraint}. 
For any coarse grid block $K$,  we define a bubble function $B$ on $K$ such that 
$B(x) = 0$ for all $x \in \partial K$ and $B(x) > 0$ for all $x \in \text{int}\left(K\right)$. 
More precisely, we take $B =  \prod_j \chi_j^{\ms}$, where the product is taken 
over all the coarse grid nodes lying on the boundary $\partial K$. 
We define a constant $C_{\pi}$ such that 
$$
C_\pi := \sup_{K \in \mathcal{T}^H,~ \mu \in V_{\text{aux}}} 
\dfrac{\int_K \tilde{\kappa} \mu^2}{\int_K \tilde{\kappa} B \mu^2}.
$$
Furthermore, we will make use of the fine-scale Lagrange interpolation operator $I_h$ defined as 
$$I_h: C^0(\Omega) \cap H_0^1(\Omega) \to C^0(\Omega) \cap V_h$$ such that 
for all $u \in C^0(\Omega) \cap H_0^1(\Omega)$, 
the interpolant $I_h u \in C^0(\Omega) \cap V_h$ is a piecewise bilinear polynomial in each fine block $\tau \in \mathcal{T}^h$ given by
\begin{equation}
(I_h u)(x) := u(x) \quad \text{ for all vertices } x \in \tau,
\end{equation}
which satisfies the standard approximation properties: 
there exists $C_I \geq 1$ such that for any $u \in C^0(\Omega) \cap H_0^1(\Omega)$, 
\begin{eqnarray}
\begin{split}
\left\| \tilde{\kappa}^\frac{1}{2} (u - I_h u) \right\|_{L^2(\tau)} + h \left\| \kappa^\frac{1}{2} \nabla \left(u - I_h u\right) \right\|_{L^2(\tau)} &\leq C_{I} h \left\| \kappa^\frac{1}{2} \nabla u \right\|_{L^2(\tau)},\\
\norm{u - I_h u}_{L^2(e)} &\leq C_I h^2 \norm{u}_{L^2(e)},
\end{split}
\label{eq:interpolation}
\end{eqnarray}
on each fine edge $e \subset \partial \tau$ and each fine block $\tau \in \mathcal{T}^h$. 
We assume that 
the following smallness criterion on the fine mesh size $h$ holds; that is, we have 
\begin{equation}
C_\pi C_I (C_\mathcal{T}^2 + \lambda_{\max}) \| \Theta \|_{L^\infty(\Omega)}^\frac{1}{2} h < 1,
\label{eq:fine-small}
\end{equation}
where $C_\mathcal{T}$ is the maximum number of vertices over all coarse elements $K \in \mathcal{T}^H$ and
\begin{equation}
\lambda_{\max} := \max_{1 \leq i \leq N} \lambda_{L_i}^{(i)}, \quad 
\Theta := \sum_{j=1}^{N_c} \vert \nabla \chi_j \vert^2. 
\end{equation}

We first recall the following theoretical result from \cite{cheung2020constraint} that is useful for our analysis of online adaptive method. 
\begin{lemma}[Lemma 2 in \cite{cheung2020constraint}]
\label{lemma:1-on}
Assume the following smallness criterion \eqref{eq:fine-small} holds. 
For any $v_{\text{aux}} \in V_{\text{aux}}$, there exists a function $v \in C^0(\Omega) \cap V_h$ such that 
\begin{equation}
\pi(v) = v_{\text{aux}}, \quad 
\| v \|_a^2 \leq D \| v_{\text{aux}} \|_s^2, \quad
\text{supp}(v) \subseteq \text{supp}(v_{\text{aux}}),
\end{equation}
where the constant $D$ is defined by 
\begin{equation}
D := \left(\dfrac{2C_\pi(1+C_I^2) \left( C_\mathcal{T}^2 + \lambda_{\max} \right)}
{1 - C_\pi C_I \left( C_\mathcal{T}^2+ \lambda_{\max} \right) \| \Theta\|_{L^\infty(K_i)}^\frac{1}{2} h}\right)^2.
\end{equation}
\end{lemma}

Throughout this section, to simplify notations, we write $a \lesssim b$ if there exists a generic constant $C$ such that $a \leq Cb$. 
The first part of our analysis is devoted to provide an error estimate for the offline coarse-scale IPDG scheme \eqref{eq:sol_ms}. 
To this end, we will justify the construction of local multiscale basis function defined in \eqref{eq:var2}. 
As we will show, the local multiscale basis function is an approximation the corresponding global basis function 
defined as follows: find $\psi_j^{(i)} \in V_h$ such that 
\begin{equation}
a_{\dg}\left(\psi_{j}^{\left(i\right)}, \psi\right) + s\left(\pi\left(\psi_{j}^{\left(i\right)}\right), \pi\left(\psi\right)\right) = s\left(\phi_j^{\left(i\right)},\pi\left(\psi\right)\right) \quad \text{ for all } \psi \in V_h.
\label{eq:var2_glo}
\end{equation}
We define the global multiscale space as $V_{\text{glo}} := \text{span} \{ \psi_j^{(i)}: 1 \leq j \leq L_i,  1 \leq i \leq N \}$. The global basis functions have a property of exponential decay and it motivates the localization and the use of the localized multiscale basis functions $\psi_{j,\text{ms}}^{(i)}$. 
We denote $\tilde V_h$ the kernel of the operator $\pi$. We remark that for any $\psi_j^{(i)} \in V_{\glo}^{(i)}$, we have 
$$ a_{\dg} (\psi_j^{(i)} , v) = 0 \quad \tforall v\in \tilde V_h,$$
which implies $\tilde V_h \subset V_{\glo}^{\perp_a}$, where $V_{\glo}^{\perp_a}$ is the orthogonal complement of $V_{\glo}$ with respect to the bilinear form $a_{\dg} (\cdot,\cdot)$. Moreover, since the dimension of the multiscale space $V_{\glo}$ is equal to that of the auxiliary space $V_{\text{aux}}$, we have $\tilde V_h = V_{\glo}^{\perp_a}$ and thus $V_h = V_{\glo} \oplus \tilde V_h$. 
The following result indicates that the global basis function defined in \eqref{eq:var2_glo} has a property of exponential decay outside an oversampled region. 
This result motivates the construction of local multiscale basis function defined in \eqref{eq:var2}. Furthermore, sufficiently many auxiliary basis functions should be included in order to ensure a fast exponential decay. 
The proof of this lemma is given in Appendix \ref{sec:appen}. 

\begin{lemma} 
\label{lemma:2-on}
Let $m \geq 2$ be an integer. Denote $K_i^+ = K_{i,m}$ an oversampled region extended from 
each coarse grid block $K_i \in \mathcal{T}^H$. 
Let $\psi_j^{\left(i\right)} \in V_{\text{glo}}$ be the global multiscale basis function 
obtained from \eqref{eq:var2_glo}, and 
$\psi_{j,\ms}^{\left(i\right)} \in V_h\left(K_{i,m}\right)$ be the localized multiscale basis function 
obtained from \eqref{eq:var2}.
Then, there exists a generic constant $C_g >0$ independent of the coarse mesh size $H$ and $\kappa_1$ such that 
\begin{equation}
\norm{\psi_j^{\left(i\right)} - \psi_{j,\ms}^{\left(i\right)}}_a^2 + \norm{\pi (\psi_j^{\left(i\right)} - \psi_{j,\ms}^{\left(i\right)})}_s^2  \leq C_g E_m \left (  \norm{\psi_j^{\left(i\right)}}_{a}^2 + \norm{\pi (\psi_j^{(i)})}_s^2 \right ),
\end{equation}
where $E_m :=  (1+\Lambda^{-1}) \left(1 +(1+\Lambda^{-1} )^{-1} \right)^{1-m}$ is the factor of exponential decay corresponding to the number $m$ of oversampling layers, and $\Lambda = \min_{1 \leq i \leq N} \lambda^{(i)}_{L_i +1}$. 
\end{lemma}

Next, we will need the following lemma in our analysis. The proof of this result is given in Appendix \ref{sec:appen}. 
\begin{lemma}
\label{lemma:3-on}
Assume that the same conditions in Lemma \ref{lemma:2-on} hold. Then, we have 
\begin{equation}
\begin{split}
& \norm{\sum_{i=1}^N \sum_{j=1}^{L_i} c_j^{(i)}  ( \psi_j^{(i)} - \psi_{j,\ms}^{(i)} ) }_a^2 + \norm{\sum_{i=1}^N \sum_{j=1}^{L_i} c_j^{(i)}  \pi ( \psi_j^{(i)} - \psi_{j,\ms}^{(i)} ) }_s^2 \\
& \lesssim (1+\Lambda^{-1}) \sum_{i=1}^N \left (\norm{ \sum_{j=1}^{L_i} c_j^{(i)}  ( \psi_j^{(i)} - \psi_{j,\ms}^{(i)} ) }_a^2 + \norm{\sum_{j=1}^{L_i} c_j^{(i)} \pi  ( \psi_j^{(i)} - \psi_{j,\ms}^{(i)} )}_s^2 \right )
\end{split}
\end{equation}
for any set of numbers $\left \{ c_j^{(i)} \right \}$. 
\end{lemma}

The remaining of this section is devoted to provide an error estimate for the 
online adaptive algorithm in Algorithm \ref{algo:online}. 
Similar to the multiscale basis functions in \eqref{eq:var2}, 
the online basis function in \eqref{eqn:loc_online_basis} is a localization 
for the corresponding global online basis function: find $\beta_{\text{glo}}^{(i)} \in V_h$ such that 
\begin{equation}
a_{\dg} (\beta_{\text{glo}}^{(i)}, v) + s(\pi (\beta_{\text{glo}}^{(i)}), \pi(v)) = r_i(v) \quad \text{ for all } v \in V_h.
\label{eqn:glo_online_basis}
\end{equation}
As stated in the following lemma, similar localization results holds for the online basis functions. 
The proof is the same as that of Lemma \ref{lemma:2-on} and 
Lemma \ref{lemma:3-on} and is therefore omitted. 

\begin{lemma}
\label{lemma:4-on}
Assume that the same conditions in Lemma \ref{lemma:2-on} hold.
Let $\beta_{\text{on}}^{(i)}$ be the online basis functions defined in \eqref{eqn:loc_online_basis} and $\beta_{\text{glo}}^{(i)}$ be the global online basis functions defined in \eqref{eqn:glo_online_basis}. Then, we have 
\begin{equation}
\norm{\beta_{\text{glo}}^{\left(i\right)} - \beta_{\text{on}}^{\left(i\right)}}_a^2 + \norm{\pi (\beta_{\text{glo}}^{\left(i\right)} - \beta_{\text{on}}^{\left(i\right)})}_s^2  \lesssim E_m \left (  \norm{\beta_{\text{glo}}^{\left(i\right)}}_{a}^2 + \norm{\pi (\beta_{\text{glo}}^{(i)})}_s^2 \right ),
\end{equation}
where $E_m$ is the factor of exponential decay in Lemma \ref{lemma:2-on}. Furthermore, we have 
\begin{equation}
\begin{split}
& \norm{\sum_{i=1}^N \sum_{j=1}^{L_i} c_j^{(i)}  ( \beta_{\text{glo}}^{(i)} - \beta_{\text{on}}^{(i)} ) }_a^2 + \norm{\sum_{i=1}^N \sum_{j=1}^{L_i} c_j^{(i)}  \pi ( \beta_{\text{glo}}^{(i)} - \beta_{\text{on}}^{(i)} ) }_s^2 \\
& \lesssim (1+\Lambda^{-1}) \sum_{i=1}^N \left (\norm{ \sum_{j=1}^{L_i} c_j^{(i)}  ( \beta_{\text{glo}}^{(i)} - \beta_{\text{on}}^{(i)} ) }_a^2 + \norm{\sum_{j=1}^{L_i} c_j^{(i)} \pi  ( \beta_{\text{glo}}^{(i)} - \beta_{\text{on}}^{(i)} )}_s^2 \right )
\end{split}
\end{equation}
for any set of numbers $\left \{ c_j^{(i)} \right \}$. 
\end{lemma}

We are going to present the main result for the online adaptive enrichment algorithm. 
Define a constant $C_{\text{poin}} \in \mathbb{R}$ such that 
\begin{equation}
C_{\text{poin}} := \sup_{v \in V_h} \frac{\norm{\pi(v)}_s}{\norm{v}_a}.
\label{eqn:poin-const}
\end{equation}
It is remarkable that $C_{\text{poin}}^2 \leq \max \{\tilde \kappa\} C_p^2$, where $C_p$ is the Poincar\'{e} constant defined by $\norm{w}_{L^2(\Omega)} \leq C_p \norm{\nabla w}_{L^2(\Omega)}$ for $w \in V_h$. 

\begin{theorem}
\label{thm:main}
Let $u_h$ be the solution of \eqref{eq:sol_dg} and let $ \{ u_{\ms}^{(k)} \}_{k=0}^\infty$ be the sequence of multiscale solutions obtained by the online adaptive enrichment algorithm. Then, we have the following error estimate: 
$$ \norm{u_h - u_{\ms}^{(k+1)} }_a^2 \lesssim (1+\Lambda^{-1})  \left [ E_m \left ( (1+D) (1+C_{\text{poin}}^2)(m+1)^d + M \right ) + M^2 \theta \right ] 
\norm{u_h - u_{\ms}^{(k)}}_a^2$$
for any integer $k \geq 0$. 
Here, $E_m$ is the factor exponential decay in Lemma \ref{lemma:2-on}, $D$ is the number given in Lemma \ref{lemma:1-on}, the number $M =\max \{ M_1, M_2\}$, where $M_1:= \max_{K \in \mathcal{T}^H} n_K$ with $n_K$ being the number of coarse nodes of the coarse element $K$, $M_2 := \max_{E \in \mathcal{E}^H} n_E$ with $n_E$ being the number of coarse neighborhoods consisting of $E$, $C_{\text{poin}}$ is the constant defined in \eqref{eqn:poin-const}, and $\theta$ is an user-defined parameter in the online adaptive enrichment algorithm. 
\end{theorem}
\begin{proof}
The proof of the convergence analysis is similar to that of the one in the CG setting \cite{chung2018fast}. 
While one has to make use of the localization estimate for the relaxed version of CEM-DG basis functions stated in Lemma \ref{lemma:2-on} to show the desired convergence estimate. 

By Ce\'{a}'s Lemma, we have 
\begin{eqnarray}
\norm{u_h - u_{\ms}^{(k+1)}}_a \lesssim \norm{u_h - w}_a
\label{eqn:main-3}
\end{eqnarray}
for any $w \in V_{\ms}^{(k+1)}$. We would find an appropriate candidate $w \in V_{\ms}^{(k+1)}$ and estimate the term in right-hand side. 
To this aim, we define global online basis function: find $\beta_{\text{glo}}^{(i,k)} \in V_h$ such that 
$$ a_{\dg} ( \beta_{\glo}^{(i,k)} , v) + s(\pi (\beta_{\glo}^{(i,k)}), \pi(v)) = r_i^k (v) \quad \tforall v \in V_h.$$
Note that 
\begin{eqnarray*}
\begin{split}
\sum_{i=1}^{N_c} r_i^k (v) &= \sum_{i=1}^{N_c} \left ( a_{\dg} (u_{\ms}^{(k)}, \chi_i v) - \int_\Omega f \chi_i v ~ dx \right ) = \sum_{i=1}^{N_c} a_{\dg} (u_{\ms}^{(k)} - u_h, \chi_i v)  \\
& =  a_{\dg} ( u_{\ms}^{(k)} - u_h, v) 
\end{split}
\end{eqnarray*}
for any $v \in V_h$. Denote $\zeta := \sum_{i=1}^{N_c} \beta_{\glo}^{(i,k)}$. Then, we have 
\begin{equation}
a_{\dg}(\zeta, v) + s (\pi(\zeta), \pi(v)) = \sum_{i=1}^{N_c} r_i^k(v) = a_{\dg} (u_{\ms}^{(k)} - u_h, v)
\label{eqn:main-00}
\end{equation}
for any $v \in V_h$. Note that 
\begin{eqnarray}
\begin{split}
a_{\dg} (\zeta, \zeta) + s(\pi(\zeta),\pi(\zeta)) & \leq \norm{ u_h - u_{\ms}^{(k)}}_a \cdot \norm{\zeta}_a \\
&\leq \norm{ u_h - u_{\ms}^{(k)}}_a \left [ a_{\dg} (\zeta, \zeta) + s(\pi(\zeta),\pi(\zeta)) \right ]^{1/2}  \\
\implies \norm{\zeta}_a^2& \leq \norm{\zeta}_a^2 + \norm{\pi(\zeta)}_s^2 \leq \norm{ u_h - u_{\ms}^{(k)}}_a ^2.
\end{split}
\label{eqn:main-0}
\end{eqnarray}
by taking $v = \zeta$ in \eqref{eqn:main-00}. 
Moreover, since $\pi (v) = 0 $ for all $v \in \tilde V_h$, from \eqref{eqn:main-00} we have 
$$ a_{\dg} (u_h - u_{\ms}^{(k)} + \zeta, v)  = 0 \quad \tforall v \in \tilde V_h.$$
Thus, we have $u_h - u_{\ms}^{(k)} + \zeta \in \left ( \tilde V_h \right )^{\perp_a} = V_{\glo}$ and there exists a set of numbers $\{ c_j^{(i)} \}$ such that 
\begin{eqnarray}
\eta:= u_h - u_{\ms}^{(k)} + \zeta = \sum_{i=1}^N \sum_{j=1}^{L_i} c_j^{(i)} \psi_j^{(i)}.
\label{eqn:cij}
\end{eqnarray}
Let $\mathcal{I}_k$ be the set of indices for the $k$-th level of online adaptive enrichment. 
For the adaptive algorithm, we add the local online basis function $\beta_{\on}^{(i,k)}$ for $i \in \mathcal{I}_k$. We take $w$ in \eqref{eqn:main-3} such that 
$$ w = u_{\ms}^{(k)}  + \eta_{\ms} - \sum_{i \in \mathcal{I}_k} \beta_{\on}^{(i,k)} \in V_{\ms}^{(k+1)}, \quad \text{where} \quad \eta_{\ms} := \sum_{i=1}^N \sum_{j=1}^{L_i} c_j^{(i)} \psi_{j,\ms}^{(i)}.$$
Then, we have 
\begin{eqnarray}
\begin{split}
\norm{u_h - u_{\ms}^{(k+1)}}_a^2 & \leq \norm{u_h - u_{\ms}^{(k)} - \eta_{\ms} + \sum_{i \in \mathcal{I}_k} \beta_{\on}^{(i,k)}}_a^2 \\
& \lesssim \norm{u_h  - u_{\ms}^{(k)} - \eta + \sum_{i \in \mathcal{I}_k} \beta_{\on}^{(i,k)}}_a^2 + \norm{\sum_{i=1}^N \sum_{j=1}^{L_i} c_j^{(i)} ( \psi_j^{(i)} - \psi_{j,\ms}^{(i)})}_a^2  \\
& = \norm{- \sum_{i\notin \mathcal{I}_k} \beta_{\glo}^{(i,k)} + \sum_{i \in \mathcal{I}_k} (\beta_{\on}^{(i,k)} - \beta_{\glo}^{(i,k)})}_a^2 + \norm{\sum_{i=1}^N \sum_{j=1}^{L_i} c_j^{(i)} ( \psi_j^{(i)} - \psi_{j,\ms}^{(i)})}_a^2  \\
& \lesssim \underbrace{\norm{\sum_{i\notin \mathcal{I}_k} \beta_{\glo}^{(i,k)}}_a^2}_{=: \mathcal{J}_1}
+ \underbrace{\norm{\sum_{i \in \mathcal{I}_k} (\beta_{\on}^{(i,k)} - \beta_{\glo}^{(i,k)})}_a^2}_{=: \mathcal{J}_2} 
+ \underbrace{\norm{\sum_{i=1}^N \sum_{j=1}^{L_i} c_j^{(i)} ( \psi_j^{(i)} - \psi_{j,\ms}^{(i)})}_a^2}_{=:\mathcal{J}_3}. 
\end{split}
\label{eqn:j1j2j3}
\end{eqnarray}
We first estimate the term $\mathcal{J}_1$. Denote $ p := \sum_{i \notin \mathcal{I}_k} \beta_{\glo}^{(i,k)}$. 
Note that $$\nabla (\chi_i p ) = ( \nabla \chi_i) p  + \chi_i (\nabla p)$$ and $\abs{\chi_i}\leq 1$. Then, we have 
\begin{eqnarray}
\begin{split}
\norm{\chi_i p }_a^2& \lesssim \norm{\chi_i p }_{\dg}^2 \lesssim \norm{p}_{\dg}^2 + \norm{p}_{s}^2 \lesssim (1+\Lambda^{-1}) \left ( \norm{p}_a^2 + \norm{\pi(p)}_s^2 \right ).
\end{split}
\label{eqn:main-2-21-sp}
\end{eqnarray}
Using the definition \eqref{eqn:glo_online_basis} of global online basis functions, we have
\begin{eqnarray} 
a_{\dg}(p, v) + s(\pi(p) , \pi(v)) = \sum_{i \notin \mathcal{I}_k} r_i^k (v)  =  \sum_{i \notin \mathcal{I}_k} r^k (\chi_i v )\quad \tforall v \in V_h,
\label{eqn:main-4}
\end{eqnarray}
where $r^k$ denotes the global residual operator at $k$-th level of enrichment. 
Taking $v = p$ in \eqref{eqn:main-4}, we obtain 
\begin{eqnarray*}
\begin{split}
\norm{p}_a^2 + \norm{\pi(p)}_s^2 & =  \sum_{i \notin \mathcal{I}_k}r^k \left ( \chi_i p\right ) \leq  \sum_{i \notin \mathcal{I}_k} \left ( \sup_{v \in V_h(\omega_i)} \frac{\abs{r^k(v)}}{\norm{v}_a} \right ) \norm{\chi_i p}_a \\
& \lesssim (1+ \Lambda^{-1})^{1/2} \left ( \norm{p}_a^2 + \norm{\pi(p)}_s^2 \right )^{1/2} \sum_{i \notin \mathcal{I}_k} \norm{z_i^k}_{a^*} \\
& \leq M^{1/2} (1+ \Lambda^{-1})^{1/2} \left ( \norm{p}_{a}^2 + \norm{\pi(p)}_{s}^2 \right )^{1/2} \left ( \sum_{i \notin \mathcal{I}_k} \norm{z_i^k}_{a^*}^2 \right )^{1/2}.
\end{split}
\end{eqnarray*}
Then, we obtain 
\begin{eqnarray*}
\begin{split}
\mathcal{J}_1 & \lesssim M (1+ \Lambda^{-1}) \sum_{i \notin \mathcal{I}_k} \norm{z_i^k}_{a^*}^2 \leq M (1+ \Lambda^{-1}) \theta \sum_{i=1}^{N_c} \norm{z_i^k}_{a^*}^2,
\end{split}
\end{eqnarray*}
where $\theta$ is the user-defined parameter in the online adaptive algorithm. By definition of $z_i^k$, for any $v \in V_h(\omega_i)$, we have 
$$ z_i^k(v) = a_{\dg}(u_{\ms}^{(k)} - u_h, v) \leq \norm{\mathbf{1}_{\omega_i} (u_h - u_{\ms}^{(k)})}_{a} \norm {v}_a.$$
Thus, we have 
\begin{eqnarray}
\begin{split}
\mathcal{J}_1& \leq M (1+ \Lambda^{-1}) \theta \sum_{i=1}^{N_c} \norm{z_i^k}_{a^*}^2  \leq M (1+ \Lambda^{-1}) \theta \sum_{i=1}^{N_c}  \norm{\mathbf{1}_{\omega_i} (u_h - u_{\ms}^{(k)})}_{a}^2\\
& \lesssim M (1+ \Lambda^{-1}) \theta \sum_{i=1}^{N_c} \norm{\mathbf{1}_{\omega_i} (u_h - u_{\ms}^{(k)})}_{\dg}^2\\
& \leq M^2(1+ \Lambda^{-1}) \theta \norm{u_h - u_{\ms}^{(k)}}_{\dg}^2 \lesssim M^2(1+ \Lambda^{-1}) \theta \norm{u_h - u_{\ms}^{(k)}}_{a}^2
\end{split}
\label{eqn:j1}
\end{eqnarray}
Next, we estimate the terms $\mathcal{J}_2$ and $\mathcal{J}_3$.
Note that, by the definition of the global basis function \eqref{eq:var2_glo}, we have 
\begin{eqnarray}
a_{\dg} (\eta, v) + s(\pi(\eta),\pi(v)) = \sum_{i=1}^N \sum_{j=1}^{L_i} c_j^{(i)} s(\phi_j^{(i)}, \pi(v)) \quad \tforall v \in V_h.
\label{eqn:main-1}
\end{eqnarray}
We define $v_{\aux}^{(i)} := \sum_{j=1}^{L_i} c_j^{(i)} \phi_j^{(i)}$, where the coefficients $\{ c_j^{(i)} \}$ are defined in \eqref{eqn:cij}. Then, by the results of Lemma \ref{lemma:1-on}, there exists a function $v^{(i)} \in C^0(\Omega) \cap V_h$ such that $\pi(v^{(i)} ) = v_{\aux}^{(i)}$, $\text{supp}(v^{(i)}) \subset \text{supp}(v_{\aux}^{(i)}) = K_i$, and 
\begin{equation}
\norm{v^{(i)}}_a^2 \leq D \norm{v_{\aux}^{(i)}}_s^2.
\label{eqn:D-est}
\end{equation}
Taking $v = \sum_{i=1}^N v^{(i)}$ in \eqref{eqn:main-1}, we obtain 
\begin{eqnarray*}
\begin{split}
\sum_{i=1}^N \norm{v_{\aux}^{(i)}}_s^2 & = a_{\dg} \left ( \eta, \sum_{i=1}^N v^{(i)} \right )  + s\left (\pi (\eta), \sum_{i=1}^N \pi(v^{(i)}) \right ) \\
& \leq \norm{\eta}_{a} \cdot \norm{\sum_{i=1}^N v^{(i)}}_a + \norm{\pi(\eta)}_{s} \cdot \norm{\sum_{i=1}^N v_{\aux}^{(i)}}_s \\
& \leq \left ( \norm{\eta}_{a}^2 + \norm{\pi(\eta)}_{s}^2 \right )^{1/2} \left ( \norm{\sum_{i=1}^N v^{(i)}}_a^2 +  \norm{\sum_{i=1}^N v_{\aux}^{(i)}}_s^2 \right )^{1/2} \\
& \leq \left ( \norm{\eta}_{a}^2 + \norm{\pi(\eta)}_{s}^2 \right )^{1/2} \left ( \sum_{i=1}^N \norm{v^{(i)}}_a^2 + \sum_{i=1}^N \norm{v_{\aux}^{(i)}}_s^2 \right )^{1/2} \\
\end{split}
\end{eqnarray*}
since $v^{(i)} \in C^0(\Omega)$ and $\phi_j^{(i)}$ is $s$-orthogonal. 
Using the orthogonality of the eigenfunctions $\left \{ \phi_j^{(i)} \right \}$ again and invoking \eqref{eqn:D-est}, we obtain 
\begin{eqnarray*}
\begin{split}
\sum_{i=1}^N \sum_{j=1}^{L_i} \left ( c_j^{(i)} \right )^2 & = \sum_{i=1}^N \norm{v_{\aux}^{(i)}}_s^2 \leq (1+D) \left ( \norm{\eta}_a^2 + \norm{\pi(\eta)}_s^2 \right ) \leq (1+D) (1+C_{\text{poin}}^2) \norm{\eta}_a^2 \\
& \lesssim (1+D) ( 1+ C_{\text{poin}}^2) \left ( \norm{u_h - u_{\ms}^{(k)}}_a^2 + \norm{\zeta}_a^2 \right ) \\
& \lesssim (1+D) (1+C_{\text{poin}}^2) \norm{ u_h - u_{\ms}^{(k)}}_a^2.
\end{split}
\end{eqnarray*}
Here, we have used the inequality \eqref{eqn:main-0}. Note that $$ \norm{\psi_j^{\left(i\right)}}_{a}^2 + \norm{\pi (\psi_j^{(i)})}_s^2 \leq \norm{\phi_j^{(i)}}_s^2 = 1$$
for any index $i \in \{ 1, \cdots, N\}$ and $j \in \{ 1, \cdots, L_i \}$. 
By Lemmas \ref{lemma:2-on} and \ref{lemma:3-on}, we obtain 
\begin{eqnarray}
\begin{split}
\mathcal{J}_3 = \norm{\sum_{i=1}^N \sum_{j=1}^{L_i} c_j^{(i)} (\psi_j^{(i)} - \psi_{j,\ms}^{(i)})}_a^2 & \lesssim (1+\Lambda^{-1})(m+1)^d E_m  \sum_{i=1}^N \sum_{j=1}^{L_i} \left ( c_j^{(i)} \right )^2 \left (  \norm{\psi_j^{\left(i\right)}}_{a}^2 + \norm{\pi (\psi_j^{(i)})}_s^2 \right ) \\
& \leq  (1+\Lambda^{-1})(m+1)^d E_m  \sum_{i=1}^N \sum_{j=1}^{L_i} \left ( c_j^{(i)} \right )^2  \\
& \leq (1+D)(1+C_{\text{poin}}^2)(1+\Lambda^{-1}) (m+1)^d E_m \norm{u_h - u_{\ms}^{(k)}}_a^2.
\end{split}
\label{eqn:main-2-1}
\end{eqnarray}

\noindent
Next, we estimate the error of localization of online basis functions. 
Using the same argument as in \eqref{eqn:main-2-21-sp}, we have 
\begin{eqnarray}
\begin{split}
\norm{\chi_i \beta_{\glo}^{(i,k)} }_a^2
& \lesssim (1+\Lambda^{-1}) \left ( \norm{\beta_{\glo}^{(i,k)}}_a^2 + \norm{\pi(\beta_{\glo}^{(i,k)})}_s^2 \right ).
\end{split}
\label{eqn:main-2-21}
\end{eqnarray}
Using 
the definition \eqref{eqn:glo_online_basis} of the global online basis function and \eqref{eqn:main-2-21}, we have 
\begin{eqnarray}
\begin{split}
\sum_{i=1}^{N_c} \norm{\beta_{\glo}^{(i,k)}}_a^2 + \norm{\pi (\beta_{\glo}^{(i,k)})}_s^2 & = a_{\dg} \left ( u_{\ms}^{(k)} - u_h, \sum_{i=1}^{N_c}  \chi_i \beta_{\glo}^{(i,k)} \right ) \\
 & \leq \norm{ u_{\ms}^{(k)} - u_h }_a \norm{\sum_{i=1}^{N_c} \chi_i \beta_{\glo}^{(i,k)}}_a \\
 & \leq \norm{ u_{\ms}^{(k)} - u_h }_a \left ( \sum_{i=1}^{N_c} \norm{ \chi_i \beta_{\glo}^{(i,k)}}_a \right ) \\ 
 & \lesssim (1+\Lambda^{-1})^{1/2} \norm{ u_{\ms}^{(k)} - u_h }_a \left ( \sum_{i=1}^{N_c} \norm{\beta_{\glo}^{(i,k)}}_a^2 + \norm{\pi(\beta_{\glo}^{(i,k)})}_s^2 \right )^{1/2}, 
\end{split}
\label{eqn:main-2-2}
\end{eqnarray}
which implies 
$$ \sum_{i=1}^{N_c} \norm{\beta_{\glo}^{(i,k)}}_a^2 + \norm{\pi (\beta_{\glo}^{(i,k)})}_s^2 \lesssim (1+ \Lambda^{-1}) \norm{ u_{\ms}^{(k)} - u_h }_a^2.$$
By Lemma \ref{lemma:4-on}, we have 
\begin{eqnarray}
\begin{split}
\mathcal{J}_2 \leq \sum_{i=1}^{N_c} \norm{\beta_{\text{glo}}^{\left(i,k\right)} - \beta_{\text{on}}^{\left(i,k\right)}}_a^2 + \norm{\pi (\beta_{\text{glo}}^{\left(i,k\right)} - \beta_{\text{on}}^{\left(i,k\right)})}_s^2 & \lesssim (1+ \Lambda^{-1}) E_m  
\norm{ u_{\ms}^{(k)} - u_h }_a^2.
\end{split}
\label{eqn:j2}
\end{eqnarray}
Combining \eqref{eqn:j1}, \eqref{eqn:main-2-1}, and \eqref{eqn:j2}, we obtain 
\begin{eqnarray*}
\begin{split}
\norm{u_h - u_{\ms}^{(k+1)}}_a^2 & \lesssim \mathcal{J}_1 + \mathcal{J}_2 + \mathcal{J}_3 \\
& \lesssim (1+\Lambda^{-1})  \left [ E_m \left ( (1+D) (1+C_{\text{poin}}^2)(m+1)^d + M \right ) + M^2 \theta \right ] 
\norm{u_h - u_{\ms}^{(k)}}_a^2.
\end{split}
\end{eqnarray*}
This completes the proof. 
\end{proof}

\section{Numerical experiments}
\label{sec:numerical}

In this section, we present some numerical examples with high-contrast media 
to demonstrate the effectiveness and efficiency of the proposed online adaptive method. 
We set the computational domain to be $\Omega = (0,1)^2$. We partition the domain into $16 \times 16$ square elements with mesh size $H = \sqrt{2}/16$ to form the coarse grid $\mathcal{T}^H$. Next, for each coarse element, we further partition it into $16 \times 16$ square elements so that the mesh size $h$ of the fine grid $\mathcal{T}^h$ is $h = \sqrt{2}/256$. We use this fine grid to compute the reference solution $u_h$. 
In all the experiments, the IPDG penalty parameter in \eqref{eq:dg_bilinear} 
is set to be $\gamma = 4$ to ensure the coercivity of the bilinear form $a_{\dg}(\cdot,\cdot)$.
In the following examples, we compute the relative $L^2$ and energy errors as follows 
$$ e_{L^2}^k := \frac{\norm{u_h - u_{\ms}^{(k)}}_{L^2(\Omega)}}{\norm{u_h}_{L^2(\Omega)}} \quad \text{and} \quad e_{a}^k := \frac{\norm{u_h - u_{\ms}^{(k)}}_a}{\norm{u_h}_a}$$
to measure the accuracy of multiscale solutions. Also, we compute the numerical convergence rate as follows 
$$ \max_k \left ( \frac{ \norm{u_h - u_{\ms}^{(k+1)}}_a^2 }{ \norm{u_h - u_{\ms}^{(k)}}_a^2} \right )$$
where the maximum is taken over the set $\{1, \cdots, \texttt{NIter} \}$, where \texttt{NIter} is the maximum number of iterations in the online adaptive algorithm.  

\begin{example} \label{exp:1}
In the first example, we consider a highly heterogeneous permeability field $\kappa$ in the domain as shown in Figure \ref{fig:medium1}. 
The background value is $1$ (i.e., the blue region) and the contrast value in the channels and inclusions is $10^4$ (i.e., the yellow region). The permeability field is a piecewise constant function on the fine grid. 
The source function is set to be $f(x_1, x_2) = \left ((x_1-0.5)^2 + (x_2- 0.5)^2 \right )^{-1/4}$ for any $(x_1, x_2) \in \Omega$. 
We pick $2$ auxiliary basis functions per coarse element and set the number of oversampling layers to be $2$ to form the offline multiscale space. 
In the online stage, we set the number of oversampling layers to be $3$ to construct online basis functions. 
\end{example}

\begin{figure}[ht!]
\centering
\includegraphics[width=0.45\linewidth]{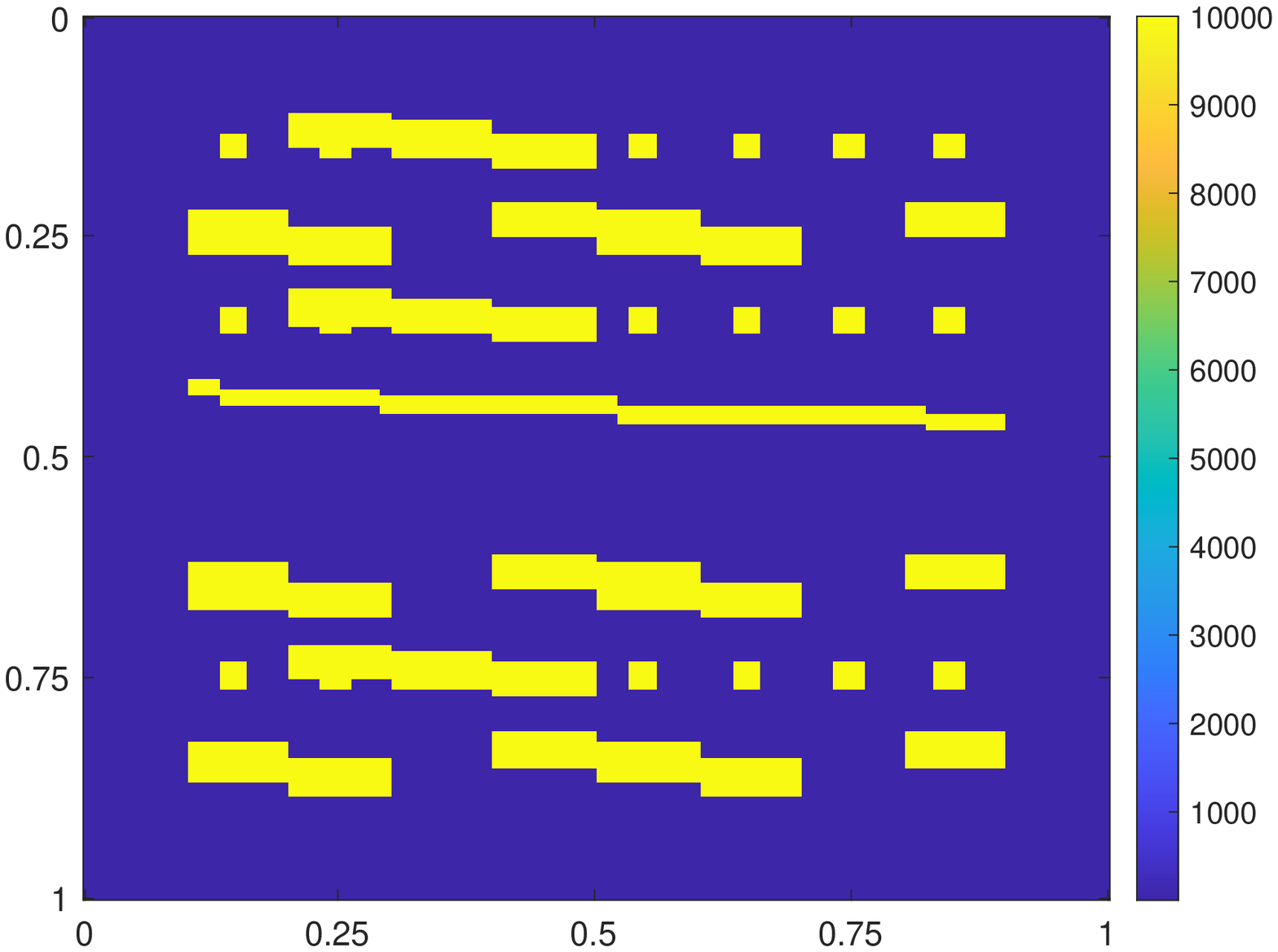} 
\caption{The permeability field $\kappa$ in Example \ref{exp:1}.}
\label{fig:medium1}
\end{figure}

We present the numerical results with different values of $\theta$ in the online adaptive algorithm. In Table \ref{tab:exp1-uni}, we present the $L^2$ and energy errors with uniform enrichment, that is $\theta = 0$. We remark that one may exclude the online basis functions whose corresponding local residuals are too small to avoid the singularity of the stiffness matrix. The column of DOFs stands for the total degrees of freedom in the current multiscale space.
In the case of uniform enrichment, one can observe a fast convergence rate; the energy error (resp. $L^2$ error) has been driven down to smaller than $0.2\%$ (resp. $0.006\%$) after three iterations with $1056$ degrees of freedom in the online multiscale space. The numerical convergence rate of uniform enrichment is $0.0568$. 

\begin{table}[ht!]
\centering
\caption{Numerical results in Example \ref{exp:1}. Convergence rate is $0.0568$ ($\theta = 0$).}
\begin{tabular}{cc|cc}
$k$-th iteration & DOFs &  $e_{L^2}^k$ & $e_{a}^k$ \\
\hline
\hline
$0$ & $512$ & $36.507255\%$ & $59.595264\%$ \\ 
$1$ & $737$ & $0.518217\%$ & $6.547061\%$ \\ 
$2$ & $940$ & $0.037000\%$ & $1.017706\%$ \\ 
$3$ & $1056$ & $0.007833\%$ & $0.242581\%$ \\ 
\end{tabular}
\label{tab:exp1-uni}
\end{table}

Next, we consider the cases of adaptive enrichment with different values of the user-defined parameter. The numerical results of error decay with $\theta = 0.3$ and $\theta = 0.6$ are presented in Tables \ref{tab:exp1-theta-03} and \ref{tab:exp1-theta-06}, respectively. That is, one may only add online basis functions for regions which account for the largest $70\%$ and $40\%$ of the residuals, respectively. One can observe from Table \ref{tab:exp1-theta-03} that the numerical convergence rate is approximately equal to $0.3079$. 
The numerical convergence rate in the case of $\theta = 0.6$ is around $0.7008$. 
This confirms the theoretical assertion that the convergence rate can be controlled by the user-defined parameter $\theta$. 
Moreover, we note that the adaptive algorithm allows adding relatively few degrees of freedom to reduce both $L^2$ and energy error to a relatively small stage. 

\begin{table}[ht!]
\centering
\caption{Numerical results in Example \ref{exp:1}. Convergence rate is $0.3079$ ($\theta = 0.3$).}
\begin{tabular}{cc|cc}
$k$-th iteration & DOFs &  $e_{L^2}^k$ & $e_{a}^k$ \\
\hline
\hline
$0$ & $512$ & $36.507255\%$ & $59.595264\%$ \\ 
$1$ & $533$ & $7.749803\%$ & $27.445575\%$ \\ 
$2$ & $557$ & $1.269979\%$ & $11.161632\%$ \\ 
$3$ & $596$ & $0.328987\%$ & $5.458823\%$ \\ 
$4$ & $653$ & $0.131134\%$ & $3.029020\%$ \\ 
\end{tabular}
\label{tab:exp1-theta-03}
\end{table}

\begin{table}[ht!]
\centering
\caption{Numerical results in Example \ref{exp:1}. Convergence rate is $0.7008$ ($\theta = 0.6$).}
\begin{tabular}{cc|cc}
$k$-th iteration & DOFs &  $e_{L^2}^k$ & $e_{a}^k$ \\
\hline
\hline
$0$ & $512$ & $36.507255\%$ & $59.595264\%$ \\ 
$1$ & $521$ & $26.211906\%$ & $49.888899\%$ \\ 
$2$ & $531$ & $14.030041\%$ & $36.631615\%$ \\ 
$3$ & $541$ & $4.574557\%$ & $21.141146\%$ \\ 
$4$ & $554$ & $1.616813\%$ & $12.649745\%$ \\ 
$5$ & $568$ & $0.723712\%$ & $8.445237\%$ \\ 
$6$ & $586$ & $0.384609\%$ & $6.027194\%$ \\ 
\end{tabular}
\label{tab:exp1-theta-06}
\end{table}


\begin{example} \label{exp:2}
In the second example, we consider a highly channelized permeability field $\kappa$ in the domain as shown in Figure \ref{fig:medium2}. In the background region (i.e., the blue region), the value of the permeability is equal to $1$. On the other hand, the value of permeability field in the channelized region (i.e., the yellow region) is equal to $10^4$. The rest of the settings (source function, number of auxiliary modes, offline and online number of oversampling layers) are the same as in Example \ref{exp:1}. 
\end{example}

\begin{figure}[ht!]
\centering
\includegraphics[width=0.45\linewidth]{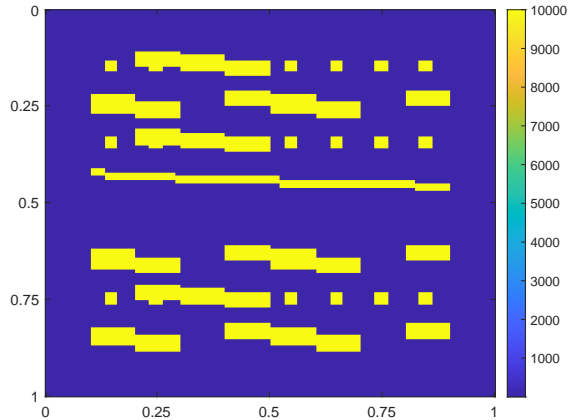} 
\caption{The permeability field $\kappa$ in Example \ref{exp:2}.}
\label{fig:medium2}
\end{figure}

We present the numerical results with uniform enrichment in Table \ref{tab:exp2-uni}. 
In this case, one can also observe a fast convergence rate as in Example \ref{exp:1}. 
Specifically, the energy error (resp. $L^2$ error) has been driven down to around $0.3\%$ (resp. $0.008\%$) after three iterations with $887$ degrees of freedom in the multiscale space. Here, we have applied the technique of singularity protection to avoid adding basis functions corresponding to extremely small local residuals. 
The numerical convergence rate of uniform enrichment is $0.1463$. 

\begin{table}[ht!]
\centering
\caption{Numerical results in Example \ref{exp:2}. Convergence rate is $0.1463$ ($\theta = 0$).}
\begin{tabular}{cc|cc}
$k$-th iteration & DOFs &  $e_{L^2}^k$ & $e_{a}^k$ \\
\hline
\hline
$0$ & $512$ & $46.799045\%$ & $66.972703\%$ \\ 
$1$ & $737$ & $0.533127\%$ & $7.126747\%$ \\ 
$2$ & $829$ & $0.013156\%$ & $0.807402\%$ \\ 
$3$ & $887$ & $0.008870\%$ & $0.308805\%$ \\ 
\end{tabular}
\label{tab:exp2-uni}
\end{table}

We consider the adaptive enrichment with different values of the user-defined parameter $\theta$ for the channelized case. The numerical results of error decay with $\theta = 0.3$ and $\theta = 0.6$ are presented in Tables \ref{tab:exp2-theta-03} and \ref{tab:exp2-theta-06}, respectively. 
One can observe from Table \ref{tab:exp2-theta-03} that the numerical convergence rate is approximately equal to $0.332$. 
On the other hand, from Table \ref{tab:exp2-theta-06}
the numerical convergence rate in the case of $\theta = 0.6$ is around $0.6049$. 
The adaptive algorithm allows adding relatively few degrees of freedom to reduce both $L^2$ and energy error to a relatively small stage. For instance, one can reduce the error to the level around $3.5\%$ by setting $\theta = 0.3$ with additionally $100$ more basis functions. 

\begin{table}[ht!]
\centering
\caption{Numerical results in Example \ref{exp:2}. Convergence rate is $0.3329$ ($\theta = 0.3$).}
\begin{tabular}{cc|cc}
$k$-th iteration & DOFs &  $e_{L^2}^k$ & $e_{a}^k$ \\
\hline
\hline
$0$ & $512$ & $46.799045\%$ & $66.972703\%$ \\ 
$1$ & $536$ & $15.586852\%$ & $38.639358\%$ \\ 
$2$ & $556$ & $2.285647\%$ & $14.834456\%$ \\ 
$3$ & $579$ & $0.616199\%$ & $7.719475\%$ \\ 
$4$ & $617$ & $0.139157\%$ & $3.544308\%$ \\ 
\end{tabular}
\label{tab:exp2-theta-03}
\end{table}

\begin{table}[ht!]
\centering
\caption{Numerical results in Example \ref{exp:2}. Convergence rate is $0.6049$ ($\theta = 0.6$).}
\begin{tabular}{cc|cc}
$k$-th iteration & DOFs &  $e_{L^2}^k$ & $e_{a}^k$ \\
\hline
\hline
$0$ & $512$ & $46.799045\%$ & $66.972703\%$ \\ 
$1$ & $524$ & $28.303085\%$ & $52.086409\%$ \\ 
$2$ & $535$ & $15.962970\%$ & $39.149738\%$ \\ 
$3$ & $547$ & $6.981006\%$ & $25.936662\%$ \\ 
$4$ & $559$ & $3.145128\%$ & $17.428573\%$ \\ 
$5$ & $570$ & $1.194350\%$ & $10.751979\%$ \\ 
$6$ & $583$ & $0.454958\%$ & $6.656748\%$ \\ 
\end{tabular}
\label{tab:exp2-theta-06}
\end{table}

In Figure \ref{fig:dist-exp2}, we show the distributions of the online basis functions at the final iteration for the adaptive enrichment (i.e., the cases when $\theta = 0.3$ and $\theta = 0.6$). For both cases, we see that basis functions are mostly added in the high permeability region $\{ 0.375 \leq x_1 \leq 0.625\}$ 
as shown in Figure~\ref{fig:medium2}. This suggests that the algorithm successfully identifies coarse regions which are highly correlated to the inaccuracy of the current solution using the residual information and includes online basis functions which have crucial effects in correcting the solution.

\begin{figure}[ht!]
\centering
\mbox{
\includegraphics[width=0.45\linewidth]{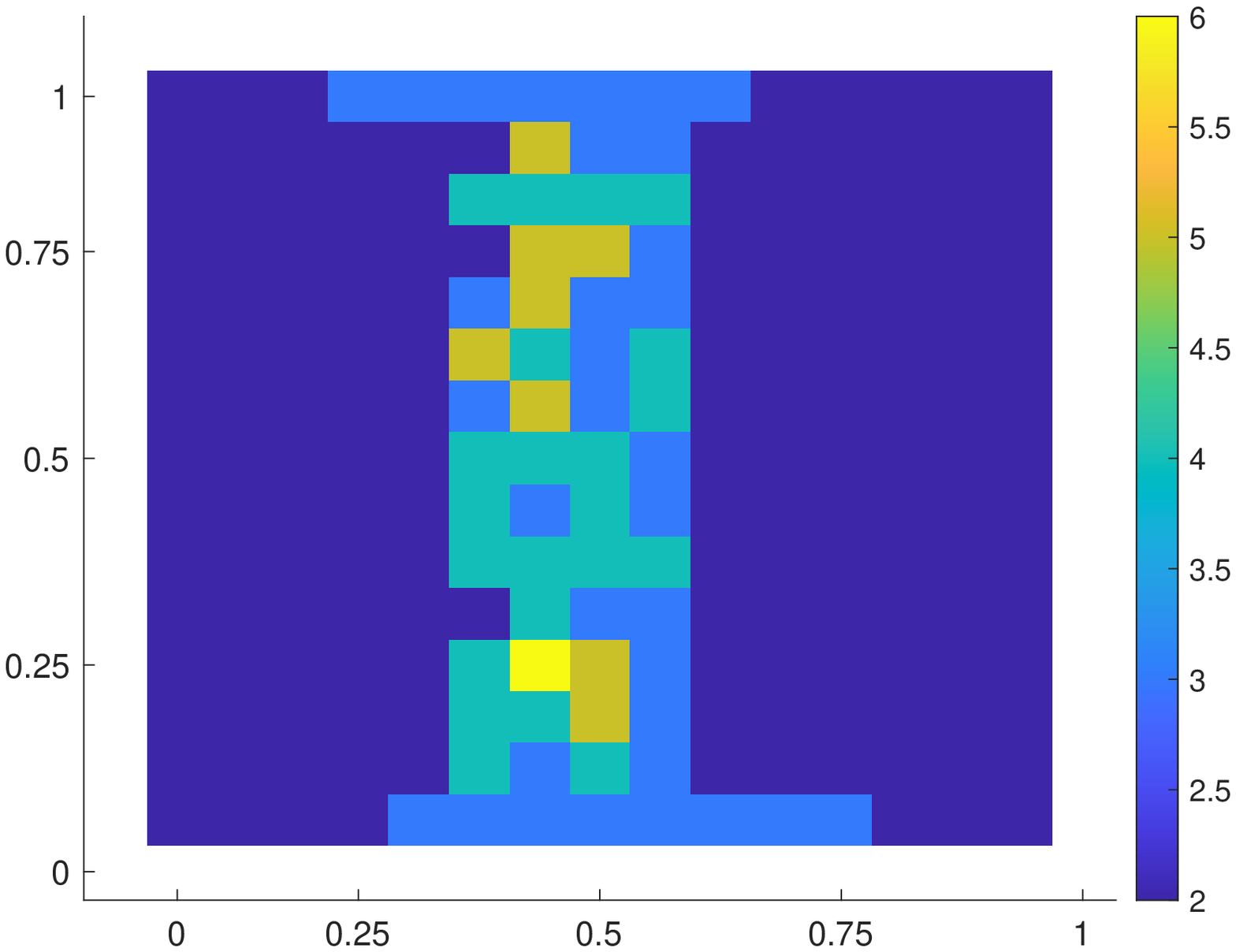}
\quad  
\includegraphics[width=0.45\linewidth]{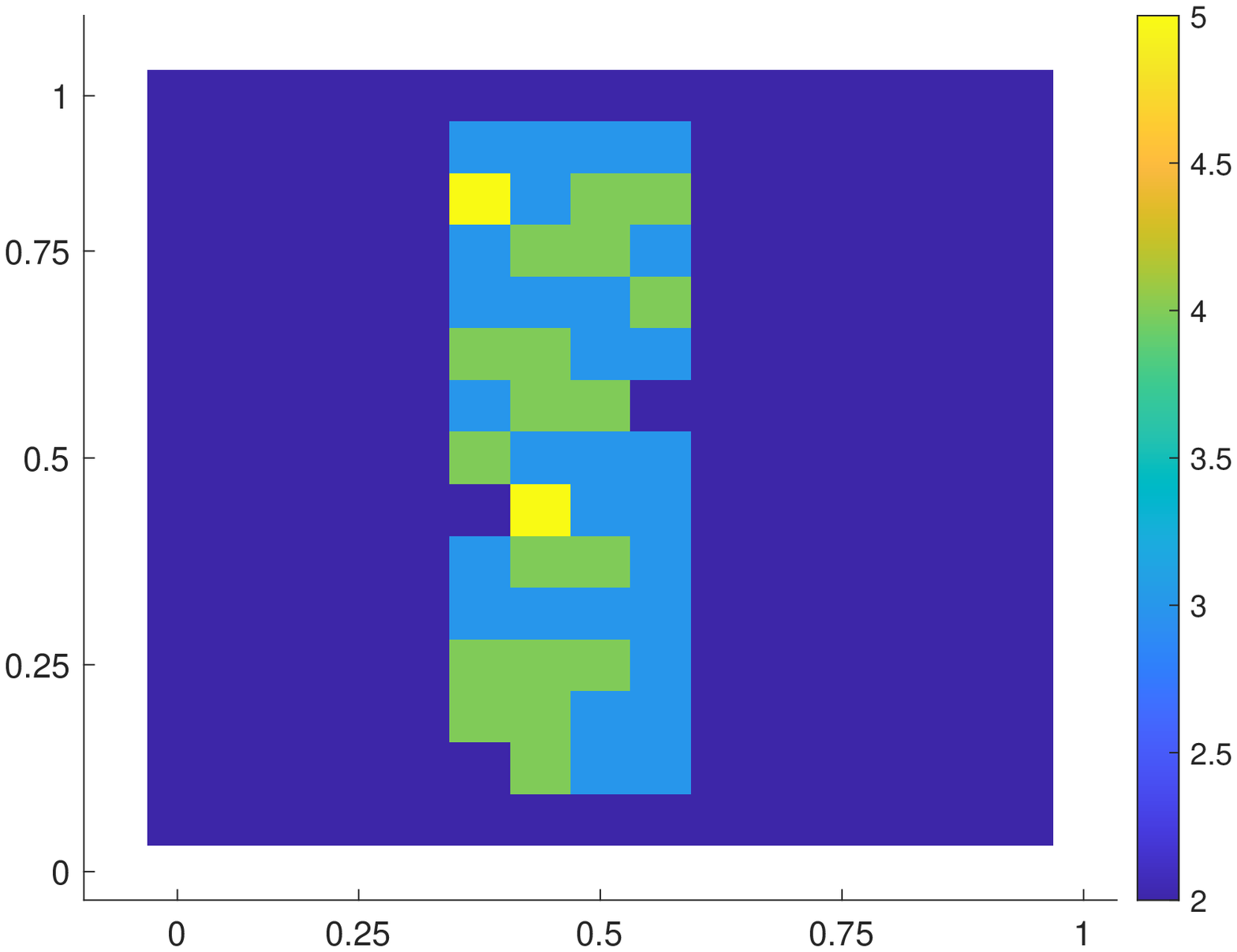}}
\caption{Distributions of online bases in Example \ref{exp:2}: $\theta = 0.3$ (left) and $\theta = 0.6$ (right).}
\label{fig:dist-exp2}
\end{figure}

\section{Conclusion}
\label{sec:conclusion}
In this research, we proposed an online adaptive strategy within the framework of the CEM-GMsDGM. 
The CEM-GMsDGM developed in \cite{cheung2020constraint} provides a systematic approach to construct offline multiscale basis functions that give a first-order convergence rate with respect to the coarse mesh size. 
The convergence rate is independent of the underlying heterogeneous media in the problem. 
In this work, 
the proposed algorithm gave a flexible approach to enrich multiscale degrees of freedom for error reduction during the online stage without any additional mesh refinement of the domain. 
The construction of  online basis functions was based on the oversampling technique and the information of local residuals with respect to the current multiscale approximation. 
The analysis showed that the convergence rate of the online adaptive algorithm depends on the factor of exponential decay and a user-defined parameter. 
Numerical experiments were provided to validate the analytical estimate. In the future, we will extend this method to wave equations and convection-diffusion equations. 

\section*{Acknowledgement}
This work was performed under the auspices of the U.S. Department of Energy by Lawrence Livermore National Laboratory under Contract DE-AC52-07NA27344.

\bibliographystyle{plain}
\bibliography{references}

\appendix 
\section{Proofs of technical results} \label{sec:appen}
In this appendix, we prove the Lemmas \ref{lemma:2-on} and \ref{lemma:3-on}. To this aim, we define a class of cutoff functions that will be used below. For each coarse block $K_i \in \mathcal{T}^H$, we define $\chi_i^{\mathcal{M},m} \in \text{span} \{ \chi_j \}_{j=1}^{N_c}$ such that 
$$ 0 \leq \chi_i^{\mathcal{M},m} \leq 1 \quad \text{and} \quad \chi_i^{\mathcal{M},m} = \left \{ \begin{array}{cl} 
1 & \text{in} ~ K_{i,m}, \\
0 & \text{in} ~ \Omega \setminus K_{i,\mathcal{M}}. \end{array} \right .$$
Here, $m$ and $\mathcal{M}$ are two given nonnegative integers with $m<\mathcal{M}$. 
Moreover, we define the following DG norm on $K_{i,\mathcal{M}} \setminus K_{i,m}$ formed by a union of coarse blocks $K \in \mathcal{T}^H$ 
$$\norm{w}_{\dg(K_{i,\mathcal{M}} \setminus K_{i,m})} := \left ( \sum_{K_i \subset K_{i,\mathcal{M}} \setminus K_{i,m}} a_i(w,w)  + \sum_{E \in \mathcal{E}^H (K_{i,\mathcal{M}} \setminus K_{i,m}) } \int_E \overline{\kappa} \llbracket w \rrbracket^2 \, d\sigma \right )^{1/2},$$
where $\mathcal{E}^H (K_{i,\mathcal{M}} \setminus K_{i,m})$ denotes the collection of all coarse edges in $\mathcal{E}^H$ which lie within the interior of $K_{i,\mathcal{M}} \setminus K_{i,m}$ and the boundary of $K_{i,\mathcal{M}}$. 
To simplify notations, we write 
$$ \norm{\psi}_{\text{CEM}} := \left ( \norm{\psi}_{\dg}^2 + \norm{\pi(\psi)}_s^2 \right )^{1/2} \quad \text{and} \quad
\norm{\psi}_{\text{CEM}(\Omega')} := \left ( \norm{\psi}_{\dg(\Omega')}^2 + \norm{\pi(\psi)}_{s(\Omega')}^2 \right )^{1/2}, $$ 
where $\Omega' \subseteq \Omega$ is formed by a union of coarse blocks $K \in \mathcal{T}^H$.

\subsection{Proof of Lemma \ref{lemma:2-on}}
Subtracting \eqref{eq:var2} from \eqref{eq:var2_glo}, we have 
$$ a_{\dg} (\psi_j^{(i)} - \psi_{j,\ms}^{(i)} , v ) + s(\pi (\psi_j^{(i)} - \psi_{j,\ms}^{(i)}), \pi(v) ) = 0 \quad \text{for any}~ v \in V_h(K_{i,m}).$$ 
Taking $v = w- \psi_{j,\ms}^{(i)}$ with $w \in V_h(K_{i,m})$, one can show that 
\begin{eqnarray}
\norm{\psi_{j}^{(i)}  - \psi_{j,\ms}^{(i)}}_a^2 + \norm{\pi(\psi_{j}^{(i)}  - \psi_{j,\ms}^{(i)})}_s^2  \leq \norm{\psi_{j}^{(i)}  -w}_a^2 + \norm{\pi(\psi_{j}^{(i)}  -w)}_s^2
\label{eqn:lem2-1}
\end{eqnarray}
for any $w \in V_h(K_{i,m})$. 
Taking $w = I_h(\chi_i^{m,m-1} \psi_j^{(i)})$ in \eqref{eqn:lem2-1}, we obtain 
\begin{eqnarray}
\begin{split}
\norm{\psi_{j}^{(i)}  - \psi_{j,\ms}^{(i)}}_{\text{CEM}}^2 & \lesssim \norm{\psi_j^{(i)} - I_h (\chi_i^{m,m-1}\psi_{j}^{(i)})}_{\text{CEM}}^2 \\
& \lesssim \norm{(1- \chi_i^{m,m-1}) \psi_j^{(i)}}_{\text{CEM}}^2 + \norm{ (1 - I_h) \chi_i^{m,m-1} \psi_j^{(i)} }_{\text{CEM}}^2 .
\end{split}
\label{eqn:lem2-2}
\end{eqnarray}
To estimate the term $\norm{(1- \chi_i^{m,m-1})\psi_{j}^{(i)}}_{\text{CEM}}^2$, we use the fact that 
$$\nabla \left ( (1- \chi_i^{m,m-1}) \psi_j^{(i)} \right ) = - \nabla \chi_i^{m,m-1} \psi_j^{(i)} + (1- \chi_i^{m,m-1}) \nabla \psi_j^{(i)}, \quad 0 \leq 1- \chi_i^{m,m-1} \leq 1$$ to obtain 
\begin{eqnarray}
\begin{split}
\norm{(1- \chi_i^{m,m-1})\psi_{j}^{(i)}}_{\dg}^2 
& = 
 \sum_{K \in \mathcal{T}^H} \int_K \kappa \Abs{\nabla (1- \chi_i^{m,m-1})\psi_{j}^{(i)}}^2 \, dx \\
& \quad +  \dfrac{\gamma}{h} \sum_{E \in \mathcal{E}^H} \int_E \overline{\kappa} \left \llbracket (1- \chi_i^{m,m-1})\psi_{j}^{(i)} \right  \rrbracket^2 \, d\sigma \\
& \lesssim \sum_{K \in \mathcal{T}^H} \int_K \kappa \Abs{\nabla \chi_i^{m,m-1}}^2 \abs{\psi_j^{(i)}}^2 \, dx \\
& \quad +  \sum_{K \subset \Omega \setminus K_{i,m-1}} \int_K \kappa \abs{\nabla \psi_j^{(i)}}^2 \, dx + \dfrac{\gamma}{h} \sum_{E \in \mathcal{E}^H} \int_E \overline{\kappa} \left \llbracket \psi_{j}^{(i)} \right  \rrbracket^2 \, d\sigma  \\
& \leq  \norm{\psi_j^{(i)}}_{\dg(\Omega \setminus K_{i,m-1})}^2 + \norm{\psi_j^{(i)}}_{s(\Omega \setminus K_{i,m-1})}^2. 
\end{split}
\end{eqnarray}
Note that for each $K \in \mathcal{T}^H$, by \eqref{eqn:lambda-1}, we have 
\begin{eqnarray}
\begin{split}
\norm{\psi_{j}^{(i)}}_{s(K)}^2&= \norm{(1 - \pi) (\psi_{j}^{(i)})}_{s(K)}^2  + \norm{\pi(\psi_{j}^{(i)})}_{s(K)}^2\\
& \leq \Lambda^{-1} \int_K \kappa \abs{\nabla \psi_j^{(i)}}^2 dx + \norm{\pi(\psi_{j}^{(i)})}_{s(K)}^2\\
& \leq \Lambda^{-1} \norm{\psi_j^{(i)}}_{\dg(K)}^2 + \norm{\pi(\psi_{j}^{(i)})}_{s(K)}^2.
\end{split}
\label{eqn:lem2-2-0}
\end{eqnarray}
Therefore, we have 
\begin{eqnarray}
\norm{(1- \chi_i^{m,m-1})\psi_{j}^{(i)}}_{\dg}^2 \lesssim  (1+\Lambda^{-1}) \left ( \norm{\psi_j^{(i)}}_{\dg(\Omega \setminus K_{i,m-1})}^2 +\norm{\pi(\psi_j^{(i)})}_{s(\Omega \setminus K_{i,m-1})}^2 \right). 
\label{eqn:lem2-2-1}
\end{eqnarray}
Next, we estimate the term $\norm{\pi((1-\chi_i^{m,m-1})\psi_{j}^{(i)})}_s^2$. Using \eqref{eqn:lem2-2-0}, we have 
\begin{eqnarray}
\begin{split}
\norm{\pi((1-\chi_i^{m,m-1})\psi_{j}^{(i)})}_s^2 & \leq \norm{(1-\chi_i^{m,m-1})\psi_{j}^{(i)}}_s^2 \leq \norm{\psi_j^{(i)}}_{s(\Omega \setminus K_{i,m-1})}^2 \\ 
& \leq \Lambda^{-1}  \norm{\psi_j^{(i)}}_{\dg(\Omega \setminus K_{i,m-1})}^2 + \norm{\pi(\psi_j^{(i)})}_{s(\Omega \setminus K_{i,m-1})}^2.
\end{split}
\label{eqn:lem2-2-2}
\end{eqnarray}
Combining \eqref{eqn:lem2-2-1} and \eqref{eqn:lem2-2-2}, the first term on the right-hand side of the inequality \eqref{eqn:lem2-2} becomes 
\begin{eqnarray}
\norm{(1- \chi_i^{m,m-1})\psi_{j}^{(i)}}_{\text{CEM}}^2   \lesssim (1+\Lambda^{-1}) \norm{\psi_j^{(i)}}_{\text{CEM}(\Omega \setminus K_{i,m-1})}^2 .
\end{eqnarray}
Now we are going to estimate the term $\norm{ (1 - I_h) \chi_i^{m,m-1} \psi_j^{(i)} }_{\text{CEM}}^2$. Denote $$\rho := (1 - I_h) \chi_i^{m,m-1} \psi_j^{(i)}.$$ By the property of the cutoff function, we have $\rho = 0$ in $\Omega \setminus K_{i,m}$. On the other hand, since $\chi_i^{m,m-1} \equiv 1$ in $K_{i,m-1}$, we have $\rho = (1 - I_h) \psi_j^{(i)} = 0$ in $K_{i,m-1}$. 
As a result, we have $\text{supp}( \rho ) \subset K_{i,m} \setminus K_{i,m-1}$. Then, using \eqref{eq:interpolation}, we have 
\begin{eqnarray}
\begin{split}
\norm{\rho}_{\text{CEM}}^2 & = \norm{\rho}_{\dg}^2 + \norm{\pi (\rho)}_{s}^2\\
& =\sum_{K_t \subset K_{i,m} \setminus K_{i,m-1}} a_t(\rho, \rho) + \sum_{E \in \mathcal{E}^H(K_{i,m} \setminus K_{i,m-1}) } \int_E \overline{\kappa} \llbracket \rho \rrbracket^2 \, d\sigma + \norm{\pi (\rho)}_{s(K_{i,m}\setminus K_{i,m-1})}^2\\
& \lesssim C_I^2 \left ( \sum_{K_t \subset K_{i,m} \setminus K_{i,m-1}} a_t(\chi_i^{m,m-1} \psi_j^{(i)}, \chi_i^{m,m-1} \psi_j^{(i)}) \right . \\
& \quad \left . + \sum_{E \in \mathcal{E}^H(K_{i,m} \setminus K_{i,m-1}) } \int_E \overline{\kappa} \llbracket \chi_i^{m,m-1} \psi_j^{(i)} \rrbracket^2 \, d\sigma \right ) \\
& \lesssim \norm{\chi_i^{m,m-1} \psi_j^{(i)}}_{\dg(K_{i,m} \setminus K_{i,m-1})}^2 = \norm{\psi_j^{(i)}}_{\dg(\Omega \setminus K_{i,m-1})}^2.
\end{split}
\label{eqn:cutoff-int-est}
\end{eqnarray}
Consequently, the inequality \eqref{eqn:lem2-2} becomes 
\begin{eqnarray}
\norm{\psi_{j}^{(i)}  - \psi_{j,\ms}^{(i)}}_{\text{CEM}}^2   \lesssim (1+\Lambda^{-1}) \norm{\psi_j^{(i)}}_{\text{CEM}(\Omega \setminus K_{i,m-1})}^2 .
\label{eqn:lem2-3-0}
\end{eqnarray}

We estimate the term 
$$\norm{\psi_j^{(i)}}_{\text{CEM}(\Omega \setminus K_{i,m-1})}^2  = \norm{\psi_j^{(i)}}_{\dg(\Omega \setminus K_{i,m-1})}^2 +\norm{\pi(\psi_j^{(i)})}_{s(\Omega \setminus K_{i,m-1})}^2.$$ 
Denote $\xi := 1- \chi_i^{m-1,m-2}$. By the definition of \eqref{eq:var2_glo}, if we take the test function to be $w = I_h(\xi^2 \psi_j^{(i)})$, then we have 
\begin{eqnarray}
a_{\dg} (\psi_j^{(i)}, I_h(\xi^2 \psi_j^{(i)}))+ s(\pi(\psi_j^{(i)}), \pi (I_h(\xi^2 \psi_j^{(i)}))) = s(\phi_j^{(i)}, \pi(I_h(\xi^2 \psi_j^{(i)}))) = 0,
\label{eqn:lem2-3-1}
\end{eqnarray}
where the last equality follows from the fact that $\text{supp}(\phi_j^{(i)}) \cap \text{supp} ( \xi^2 ) = \emptyset$. Using the techniques of showing the formula (67) in \cite{cheung2020constraint}, one can show that 
\begin{eqnarray}
\begin{split}
\norm{\psi_j^{(i)}}_{\dg(\Omega \setminus K_{i,m-1})}^2 & \lesssim \norm{\xi \psi_j^{(i)}}_a^2 \leq a_{\dg}(\psi_j^{(i)}, \xi^2 \psi_j^{(i)}) + \norm{\psi_j^{(i)}}_{s(K_{i,m-1} \setminus K_{i,m-2})}^2 \\
& \leq a_{\dg}(\psi_j^{(i)}, I_h(\xi^2 \psi_j^{(i)})) + a_{\dg}(\psi_j^{(i)}, \xi^2 \psi_j^{(i)} - I_h(\xi^2 \psi_j^{(i)})) \\
& \quad + \norm{\psi_j^{(i)}}_{s(K_{i,m-1} \setminus K_{i,m-2})}^2. 
\end{split}
\label{eqn:lem2-3}
\end{eqnarray}
On the other hand, since $ \chi_i^{m-1,m-2} \equiv 0$ in $\Omega \setminus K_{i,m-1}$, we have 
$$
s(\pi(\psi_j^{(i)}), \pi (I_h(\xi^2 \psi_j^{(i)}))) = \norm{\pi(\psi_j^{(i)})}_{s(\Omega \setminus K_{i.m-1})}^2 + \int_{K_{i,m-1} \setminus K_{i,m-2}} \tilde \kappa \pi ( \psi_j^{(i)} )  I_h( \xi^2 \psi_j^{(i)}) dx. 
$$
Thus, we have 
\begin{eqnarray}
\begin{split}
\norm{\pi(\psi_j^{(i)})}_{s(\Omega \setminus K_{i,m-1})}^2& = s(\pi(\psi_j^{(i)}), \pi (I_h(\xi^2 \psi_j^{(i)})))  - \int_{K_{i,m-1} \setminus K_{i,m-2}} \tilde \kappa \pi ( \psi_j^{(i)} ) I_h( \xi^2 \psi_j^{(i)}) dx \\
& \leq s(\pi(\psi_j^{(i)}), \pi (I_h(\xi^2 \psi_j^{(i)}))) \\
& \quad + \norm{\pi(\psi_j^{(i)})}_{s(K_{i,m-1} \setminus K_{i,m-2})} \norm{I_h(\xi^2 \psi_j^{(i)})}_{s(K_{i,m-1} \setminus K_{i,m-2})}.
\end{split}
\label{eqn:lem2-4}
\end{eqnarray}
Using \eqref{eqn:lem2-3-1} and adding \eqref{eqn:lem2-3} and \eqref{eqn:lem2-4}, we have 
\begin{eqnarray}
\begin{split}
\norm{\psi_j^{(i)}}_{\text{CEM}(\Omega \setminus K_{i,m-1})}^2 & \leq a_{\dg}(\psi_j^{(i)}, \xi^2 \psi_j^{(i)} - I_h(\xi^2 \psi_j^{(i)})) + \norm{\psi_j^{(i)}}_{s(K_{i,m-1} \setminus K_{i,m-2})}^2\\
& \quad + \frac{1}{2} \left ( \norm{\pi(\psi_j^{(i)})}_{s(K_{i,m-1} \setminus K_{i,m-2})}^2 + \norm{I_h(\xi^2 \psi_j^{(i)})}_{s(K_{i,m-1} \setminus K_{i,m-2})}^2 \right ). \\
\end{split}
\label{eqn:lem2-5-0}
\end{eqnarray}
We first analyze the terms $a_{\dg}(\psi_j^{(i)}, \xi^2 \psi_j^{(i)} - I_h(\xi^2 \psi_j^{(i)}))$ and $\norm{I_h(\xi^2 \psi_j^{(i)})}_{s(K_{i,m-1} \setminus K_{i,m-2})}^2$. Denote 
$\tilde \rho := \xi^2 \psi_j^{(i)} - I_h(\xi^2 \psi_j^{(i)})$. 
Note that $\text{supp}(\tilde \rho )  \subset K_{i,m-1} \setminus K_{i,m-2}$. 
Using the same argument as in that of showing \eqref{eqn:cutoff-int-est}, we have 
\begin{eqnarray}
\norm{\tilde \rho}_{\dg(K_{i,m-1} \setminus K_{i,m-2})}^2 + \norm{\tilde \rho}_{s(K_{i,m-1} \setminus K_{i,m-2})}^2 \lesssim C_I^2 \norm{\xi^2 \psi_j^{(i)}}_{\dg(K_{i,m-1} \setminus K_{i,m-2})}^2.
\label{eqn:rho-tilde-est}
\end{eqnarray}
Note that $\abs{\xi} \leq 1$. Using the chain rule, we have $\nabla ( \xi^2 \psi_j^{(i)} ) = \xi^2 \nabla \psi_j^{(i)} + 2 \xi \psi_j^{(i)} \nabla \xi$
and we are able to show that 
\begin{eqnarray}
\norm{\xi^2 \psi_j^{(i)} }_{\dg(K_{i,m-1} \setminus K_{i,m-2})}^2 \lesssim \norm{\psi_j^{(i)}}_{\dg(K_{i,m-1} \setminus K_{i,m-2})}^2 + \norm{\psi_j^{(i)}}_{s(K_{i,m-1} \setminus K_{i,m-2})}^2.
\label{eqn:lem2-5-1}
\end{eqnarray}
For the first term in \eqref{eqn:lem2-5-0}, using the Young's inequality, we have 
\begin{eqnarray}
a_{\dg}(\psi_j^{(i)}, \tilde \rho) \lesssim \norm{\psi_j^{(i)}}_{\dg(K_{i,m-1} \setminus K_{i,m-2})}^2 +  \norm{\tilde \rho}_{\dg(K_{i,m-1} \setminus K_{i,m-2})}^2. 
\label{eqn:lem2-5-2}
\end{eqnarray}
Then, for the fourth term in \eqref{eqn:lem2-5-0}, we have 
\begin{eqnarray}
\norm{I_h(\xi^2 \psi_j^{(i)})}_{s(K_{i,m-1} \setminus K_{i,m-2})}^2 & \lesssim \norm{\tilde \rho}_{s(K_{i,m-1} \setminus K_{i,m-2})}^2 + \norm{\xi^2 \psi_j^{(i)}}_{s(K_{i,m-1} \setminus K_{i,m-2})}^2.
\label{eqn:lem2-5-3}
\end{eqnarray}
Combining \eqref{eqn:rho-tilde-est}, \eqref{eqn:lem2-5-1}, \eqref{eqn:lem2-5-2}, and \eqref{eqn:lem2-5-3}, the inequality \eqref{eqn:lem2-5-0} becomes 
\begin{eqnarray}
\begin{split}
\norm{\psi_j^{(i)}}_{\text{CEM}(\Omega \setminus K_{i,m-1})}^2 & \lesssim \norm{\psi_j^{(i)}}_{\text{CEM}(K_{i,m-1} \setminus K_{i,m-2})}^2 + \norm{\psi_j^{(i)}}_{s(K_{i,m-1} \setminus K_{i,m-2})}^2 \\
& \leq (1+ \Lambda^{-1}) \norm{\psi_j^{(i)}}_{\text{CEM}(K_{i,m-1} \setminus K_{i,m-2})}^2,
\end{split}
\end{eqnarray}
where the last inequality follows from \eqref{eqn:lambda-1}. 

Moreover, we note that the following inequality holds 
\begin{eqnarray*}
\begin{split}
\norm{\psi_j^{(i)}}_{\text{CEM}(\Omega \setminus K_{i,m-2})}^2 & = \norm{\psi_j^{(i)}}_{\text{CEM}(\Omega \setminus K_{i,m-1})}^2 + \norm{\psi_j^{(i)}}_{\text{CEM}(K_{i,m-1} \setminus K_{i,m-2})}^2 \\
& \gtrsim \left ( 1 + (1+\Lambda^{-1} )^{-1}  \right )  \norm{\psi_j^{(i)}}_{\text{CEM}(\Omega \setminus K_{i,m-1})}^2
\end{split}
\end{eqnarray*}
for any integer $m \geq 2$. 
Using the above inequality recursively, we obtain 
$$\norm{\psi_j^{(i)}}_{\text{CEM}(\Omega \setminus K_{i,m-1})}^2  \lesssim \left ( 1 + (1+\Lambda^{-1} )^{-1}  \right )^{1-m} \norm{\psi_j^{(i)}}_{\text{CEM}}^2. $$
Finally, the inequality \eqref{eqn:lem2-3-0} becomes 
$$\norm{\psi_{j}^{(i)}  - \psi_{j,\ms}^{(i)}}_{\text{CEM}}^2   \lesssim (1+\Lambda^{-1}) \left ( 1 + (1+\Lambda^{-1} )^{-1}  \right )^{1-m} \norm{\psi_j^{(i)}}_{\text{CEM}}^2. $$
This completes the proof due to the equivalence of the norms $\norm{\cdot}_a$ and $\norm{\cdot}_{\dg}$.

\subsection{Proof of Lemma \ref{lemma:3-on}}
Given any $\{ c_j^{(i)} \}$, we denote $w^{(i)} := \sum_{j=1}^{L_i} c_j^{(i)} (\psi_j^{(i)} - \psi_{j,\ms}^{(i)})$ and we write $w := \sum_{i=1}^N w^{(i)}$. Using the definition of the global and local multiscale basis functions in \eqref{eq:var2_glo} and \eqref{eq:var2} with test function $I_h((1- \chi_i^{m+1, m})w)$, we obtain 
\begin{eqnarray*}
\begin{split}
a_{\dg}\left ( I_h((1- \chi_i^{m+1, m})w), \sum_{j=1}^{L_i} c_j^{(i)} \psi_{j}^{(i)} \right ) + s\left ( \pi(I_h(1- \chi_i^{m+1, m})w)), \sum_{j=1}^{L_i} c_j^{(i)} \pi(\psi_{j}^{(i)}) \right ) & = 0, \\
a_{\dg}\left ( I_h((1- \chi_i^{m+1, m})w), \sum_{j=1}^{L_i} c_j^{(i)} \psi_{j,\ms}^{(i)} \right ) + s\left ( \pi(I_h(1- \chi_i^{m+1, m})w)), \sum_{j=1}^{L_i} c_j^{(i)} \pi(\psi_{j,\ms}^{(i)}) \right ) & = 0. \\
\end{split}
\end{eqnarray*}
Here, we have used the fact that $\text{supp}( \sum_{j=1}^{L_i} c_j^{(i)} \psi_{j,\ms}^{(i)} ) \subset K_{i,m}$. Subtracting the two equations above, we have 
\begin{eqnarray*}
\begin{split}
\norm{w}_a^2 + \norm{\pi(w)}_s^2 & = \sum_{i=1}^N a_{\dg}( I_h(\chi_i^{m+1,m} w), w^{(i)}) + s(\pi(I_h(\chi_i^{m+1,m} w)), \pi(w^{(i)})) \\
& \leq \sum_{i=1}^N \norm{I_h(\chi_i^{m+1,m}w)}_a \norm{w^{(i)}}_a + \norm{\pi(I_h(\chi_i^{m+1,m} w))}_s \norm{\pi(w^{(i)})}_s \\
& \leq \sum_{i=1}^N \left ( \norm{I_h(\chi_i^{m+1,m}w)}_a^2 + \norm{\pi(I_h(\chi_i^{m+1,m} w))}_s ^2  \right )^{1/2} \left ( \norm{w^{(i)}}_a^2 + \norm{\pi(w^{(i)})}_s^2 \right )^{1/2}  \\
& \lesssim (1+\Lambda^{-1})^{1/2} \left ( \norm{w}_a^2 + \norm{\pi(w)}_s^2 \right )^{1/2} \left ( \sum_{i=1}^N\norm{w^{(i)}}_a^2 + \norm{\pi(w^{(i)})}_s^2 \right )^{1/2},
\end{split} 
\end{eqnarray*}
where the last inequality follows from the same argument as showing \eqref{eqn:lem2-3-0} with the replacement of $\psi_j^{(i)}$ by $w$. 
Hence, we conclude that 
$$ \norm{w}_a^2 + \norm{\pi(w)}_s^2 \lesssim (1+\Lambda^{-1}) \sum_{i=1}^N \left ( \norm{w^{(i)}}_a^2 + \norm{\pi(w^{(i)})}_s^2 \right ).$$
\end{document}